\documentclass[english,preprint,1p,times]{elsarticle}

\journal{J. COMPUT. APPL. MATH.}

\providecommand{\tabularnewline}{\\}
\usepackage{color}
\usepackage{booktabs}
\usepackage{mathtools}
\usepackage{amsbsy}
\usepackage{amstext}
\usepackage{amssymb}
\usepackage{amsthm}
\usepackage{graphicx}
\usepackage[numbers]{natbib}
\usepackage{babel}
\usepackage{subfig}

\newtheorem{thm}{Theorem}[section]

\newtheorem{rem}[thm]{Remark}
\newtheorem{prop}[thm]{Proposition}
\newtheorem{example}[thm]{Example}

\usepackage{lineno}
\usepackage{xcolor}
\begin{document}
	
\begin{frontmatter}
	
\title{Decoupled, linear, unconditionally energy stable and charge-conservative finite element method for a inductionless magnetohydrodynamic phase-field model}

\author{Xiaorong Wang}

\address{LSEC, NCMIS, Academy of Mathematics and Systems Science, Chinese Academy of Sciences; School of Mathematical Science, University of Chinese Academy of Sciences, Beijing 100190, China.}

\author{Xiaodi Zhang\corref{correspondingauthor}}

\address{Henan Academy of Big Data, Zhengzhou University, Zhengzhou 450001, China.}

\cortext[correspondingauthor]{Corresponding author: zhangxiaodi@lsec.cc.ac.cn (X. Zhang)}

\ead{zhangxiaodi@lsec.cc.ac.cn (X. Zhang), wxr@lsec.cc.ac.cn (X. Wang)}

\begin{abstract}
In this paper, we consider the numerical approximation for a diffuse interface model of the two-phase incompressible inductionless magnetohydrodynamics problem. This model consists of Cahn-Hilliard equations, Navier-Stokes equations and Poisson equation. We propose a linear and decoupled finite element method to solve this highly nonlinear and multi-physics system. For the time variable,  the discretization is a combination of first order Euler semi-implicit scheme, several first order stabilization terms and implicit-explicit treatments for coupling terms. For the space variables, we adopt the finite element discretization, especially, we approximate the current density and electric potential by inf-sup stable face-volume mixed finite element pairs. With these techniques, the scheme only involves a sequence of decoupled linear equations to  solve at each time step. We show that the scheme is provably mass-conservative, charge-conservative and unconditionally energy stable. Numerical experiments are performed to illustrate the features, accuracy and efficiency of the proposed scheme.
\end{abstract}


\end{frontmatter}

\section{Introduction}
In recent years, the phase field method, also called the diffuse interface method, has been widely used to simulate the motion of multiphase incompressible immiscible fluids for both numerical and theoretical research \cite{anderson_diffuse-interface_1998,ding_diffuse_2007,liu_phase_2003}, such as in  fields like material sciences, fracture mechanics and  fluid mechanics. Unlike the traditional sharp interface model, the phase field method describes the interface by a balance of molecular forces in a very thin layer rather than a free curve that evolves over time. To indicate different phases, an auxiliary function $\phi$ called phase field or order parameter is introduced \cite{rayleigh_theory_1892,waals_thermodynamic_1893}. The phase function takes a distinct constant value in each bulk phase, and varies smoothly in the interfacial region. In the diffuse interface theory, the surface motion of the order parameter is driven by a gradient flow. There are two popular gradient flow-type governing equations: Allen-Cahn equations \cite{allen_1979} and Cahn-Hilliard equations \cite{cahn_free_1958}. Comparing with the classical sharp interface model, one principal advantage of phase field method is its ability to capture the interface implicitly and automatically, and allow for topological changes such as self-intersection, pinch-off and splitting. Thus, phase field method has become an attractive technique, for extensive study on the phase field approach, we refer to the recent reviews \cite{kim_phase-field_2012,shen_modeling_2011} and the references therein.

Magnetohydrodynamic (MHD) equations are used to describe the behavior of electrically conductive fluids under the influence of magnetic fields and electric currents, which couple the Navier-Stokes equations and Maxwell equations. The theoretical analysis and numerical simulation of incompressible MHD equations is an area of research currently undergoing intense study \cite{gerbeau_stabilized_2000,gerbeau_mathematical_2006,moreau_magnetohydrodynamics_1990}. However, in most industrial and laboratory cases, the magnetic Reynolds number is small, then the induced magnetic field can be neglected compared with the applied magnetic field. In these cases, Maxwell  equations are replaced by Poisson equation for the electric scalar potential, which yields the inductionless magnetohydrodynamic (IMHD) equations \cite{badia_block_2014,li_charge-conservative_2019,ni_current_2007}. It is well-known that numerical simulation
of multiphase flow is a challenging problem in computational
fluid dynamics. This problem becomes more difficult for MHD due to the interaction between multiple physical fields. The multiphase MHD flow often is involved in  a wide range of scientific and industrial problems, such as astrophysics, aluminum electrolysis, liquid metal magnetic pumps, MHD power generators, fusion reactor blankets 
\cite{abdou_2001,davidson_introduction_2001}. Thus developing fast and effective algorithms for incompressible multiphase MHD equations has great theoretical significance and applicable value.

In this paper, we focus on the two-phase incompressible IMHD problem in a general Lipschitz domain, which is frequently performed to simulate the movement in aluminum electrolysis cell  \cite{gerbeau_mathematical_2006,munger_level_2006} and explore the theoretical study and numerical experiments of liquid lithium-lead cladding \cite{abdou_us_2005}. Traditionally, most of the existing models for multiphase incompressible IMHD problem are devoted to sharp interface models, for instance, the front-tracking method \cite{samulyak_numerical_2007,zhang_direct_2018}, the volume-of-fluid method \cite{huang_3d_2002,pan_development_2012,takatani_mathematical_2007}, the level-set method \cite{ki_level_2010,xie_tracking_2007}. There are few works have focused on the two-phase incompressible IMHD equations by diffuse-interface method in the literature.
In 2014, Ding et al. \citep{Ding2014} presented a
two-phase incompressible inductionless magnetohydrodynamics model and simulated the deformation of melt interface in an aluminum electrolytic cell.
In 2020, Chen and his collaborators \citep{Chen2020} proposed a linear, decoupled, unconditionally energy stable and second order time-marching schemes to simulate the two-phase incompressible Cahn-Hilliard-IMHD (CHIMHD) conducting flow. Zhang \citep{Zhang2021} derived a diffuse interface model for IMHD fluids systematically and analyzed sharp interface limits for
different choice of the mobilities. Mao et al. \cite{mao_fully_2021} analyzed the fully discrete finite element approximation of a three-dimensional CHIMHD model and proved the well-posedness of weak solution to the phase field model by using the classical compactness method.

Studying efficient numerical methods for solving incompressible CHIMHD flows is the main focus of this article. The CHIMHD model is a complex system that involves three coupled physical processes: the phase field model (Cahn-Hilliard equation), fluid dynamics (Navier-Stokes equations), and electric field (Possion equation). Due to the highly nonlinear with multi-physics fields coupled in this system, it is relatively complicated and time-consuming to solve it by nonlinear and coupled scheme \citep{Chen2020,mao_fully_2021}. Our primary goal is to propose a decoupled and efficient fully discrete mixed finite element method to solve the system. The time discretization is an Euler semi-implicit scheme with some extra first order stabilization terms in Cahn-Hilliard equations and Poisson equation, and implicit-explicit treatments for coupling terms. The subtle time splitting technique allows us to fully decouple system into three processes at each time step. For spacial discretization, we adopt classical inf-sup stable finite element to discretize the phase field and 
chemical potential, standard velocity-pressure stable finite element to discretize the hydrodynamic unknowns, and stable face-volume finite element pairs to discretize the current density and electric potential. With these techniques, we obtain an efficient fully discrete scheme, which is decoupled, linear while still mass-conservative, charge-conservative and unconditionally energy stable. Numerical experiments are performed to demonstrate the features, accuracy and efficiency of the proposed scheme.

We emphasize that the proposed scheme is charge-conservative and unconditionally energy stable. The charge conservation law, namely, $\rm{div}\boldsymbol{J} = 0$ is a physical law in electromagnetics, which plays an important role in keeping the calculation accuracy for the simulation of MHD fluid.  It was believed that
only when the numerical schemes numerically maintains the physical conservation laws, such as momentum conservation  and charge conservation, can we get accurate results for MHD flows at high Hartmann number \cite{ni_consistent_2012,ni_current_2007}. Due to the rapid changes near the interface, the non-compliance of energy dissipation laws of the scheme may lead to spurious numerical solutions \citep{Hua2011,Johnston2004}. Thus, the design of unconditionally energy stable schemes are of great importance for solving phase field models. To the best of our knowledge, the scheme presented in this paper is the first decoupled, linear and unconditionally energy stable scheme for a phase field model of two-phase inductionless MHD flows.

The rest of this paper is organized as follows. In section 2, we present the CHIMHD model, derive a weak formulation and the dissipative energy law. In section 3, we propose a fully discrete charge-conservation finite element
method and prove its unconditional energy stability. In section 4, we present some numerical simulations to validate our schemes. Some concluding remarks are given in section 5.

\section{The diffuse interface model}\label{sec:model}
In this section, we first introduce the CHIMHD model, then present a weak formulation, and finally show the dissipative energy law in the PDE level.
\subsection{Model system}
We consider a phase-field model for mixture of two immiscible, incompressible and conducting fluids with the same density in a bounded domain $\Omega$ with Lipschitz-continuous boundary $\Sigma\coloneqq\partial\Omega$
in $\mathbb{R}^{{\rm d}},{\rm d}=2,3$. Following \citep{mao_fully_2021,Zhang2021}, the dynamic behavior of the
flow is governed by the coupling model of the Cahn-Hilliard equations and the IMHD equations. In this paper, we consider the incompressible CHIMHD system as follows:
\begin{align}
\rho\left(\boldsymbol{u}_{t}+\boldsymbol{u}\cdot\nabla\boldsymbol{u}\right)-2\nabla\cdot\left(\eta\left(\varphi\right)D\left(\boldsymbol{u}\right)\right)+\nabla p+\varphi\nabla\mu & =\boldsymbol{J}\times\boldsymbol{B}\text{ in }\Omega,\label{model:fluid}\\
\mathrm{div}\boldsymbol{u} & =0\text{ in }\Omega,\label{model:divu}\\
\sigma\left(\varphi\right)^{-1}\boldsymbol{J}+\nabla\phi-\boldsymbol{u}\times\boldsymbol{B} & =\boldsymbol{0}\text{ in }\Omega,\label{model:poisson}\\
\mathrm{div}\boldsymbol{J} & =0\text{ in }\Omega,\label{model:divJ}\\
\varphi_{t}+{\rm div}(\varphi\boldsymbol{u})-M\Delta\mu & =0\text{ in }\Omega,\label{model:ch}\\
-\gamma\varepsilon\Delta\varphi+\frac{\gamma}{\varepsilon}f(\varphi)-\mu & =0\text{ in }\Omega.\label{model:chc}
\end{align}\label{model:chimhd}
where $\rho$ is the density of fluids, $\boldsymbol{u},p$ denote
the velocity and pressure of fluids, $\boldsymbol{J},\phi$ represent the current density and electric scalar potential, $\varphi,\mu$ denote the phase function and chemical potential, and $\boldsymbol{B}$ is the applied magnetic field which is assumed to be given. $D(\mathbf{u})=\frac{1}{2}(\nabla\mathbf{u}+\nabla\mathbf{u}^{\mathrm{T}})$ is the symmetric gradient tensor. In the Cahn-Hilliard equations (\ref{model:ch})-(\ref{model:chc}), $f(\varphi)=F^{\prime}(\varphi)$, with $F(\varphi)$ is the double-well potential, $\gamma$
is the surface tension coefficient, and $\epsilon$ is the interface thickness, $M$ represents the diffusional mobility related to the relaxation time scale. In the IMHD equations (\ref{model:fluid})-(\ref{model:divJ}),
the viscosity of the fluids $\eta(\varphi),$ and the electric conductivity $\sigma(\varphi)$ are depending on the phase function $\varphi$. They are Lipschitz-continuous functions of $\varphi$ satisfying
\noindent
\[
\begin{array}{l}
0<\min\left\{ \eta_{1},\eta_{2}\right\} \leq\eta(\varphi)\leq\max\left\{ \eta_{1},\eta_{2}\right\} ,\\
0<\min\left\{ \sigma_{1},\sigma_{2}\right\} \leq\sigma(\varphi)\leq\max\left\{ \sigma_{1},\sigma_{2}\right\} ,
\end{array}
\]
where $\eta_{i}$ and $\sigma_{i}$ (i=1,2) are the viscosity and electric
conductivity of the fluid on each phase. The system of equations
are supplemented with the following initial values and boundary conditions
\[
\begin{array}{r}
\boldsymbol{u}(0)=\boldsymbol{u}_{0},\quad\varphi(0)=\varphi_{0}\quad\text{ in }\Omega,\\
\boldsymbol{u}=0,\quad\boldsymbol{J}\cdot\boldsymbol{n}=0,\quad\partial_{\boldsymbol{n}}\varphi=\partial_{\boldsymbol{n}}\mu=0\quad\text{ on }\Sigma,
\end{array}
\]
where $\boldsymbol{n}$ is the outer unit normal of $\Sigma$. In
this paper, we consider two-phase flows with matching density $\rho_{1}=\rho_{2}$. Without lose of generality, we set $\rho\equiv1$ in the above system
(\ref{model:chimhd}).

\begin{rem}
The first term on the right-hand side of (\ref{model:fluid}), $\boldsymbol{J}\times\boldsymbol{B}$
is the Lorentz force. The last term on the left-hand side of (\ref{model:fluid}) the continuum surface tension force in the potential form, which denotes the phase introduced force. This term appears differently in 
literature \cite{shen_decoupled_2015,shen_modeling_2011}: $\nabla\cdot(\nabla\phi\otimes\nabla\phi),$ $-\mu\nabla\phi$
or $\phi\nabla\mu.$ It can be shown that these three expressions are
equivalent by redefining the pressure $p$.
\end{rem}

\begin{rem}Though the model considered here is a two-phase model with matched density, one can still employ this model with Boussinesq approximation to model the effect of density difference by a gravitational force in case of small density ratio. The large density ratio case is reserved for future work.
\end{rem}

In this paper, we consider the truncated double well potential $F(\varphi)$,
\[
F(\varphi)=\frac{1}{4}\left\{ \begin{array}{ll}
4(\varphi+1)^{2} & \text{ if }\varphi<-1, \\
\left(\varphi^{2}-1\right)^{2} & \text{ if }-1\leq\varphi\leq 1, \\
4(\varphi-1)^{2} & \text{ if }\varphi>1.
\end{array}\right.
\]
The original definition of the potential is logarithmic \citep{cahn_free_1958}, which guarantees the phase field $\varphi$ stays within (-1,1). The Ginzburg-Landau potential $F_{\rm{GL}}(\varphi)=\frac{1}{4}(1-\varphi^2)^2$ is an popular alternative to approximate the one suggested by Cahn and Hilliard. Following \cite{shen_decoupled_2015,shen_numerical_2010}, we further restrict the growth of the potential to quadratic away from the range [-1,1]. It is proved by \cite{caffarelli_bound_1995} that the truncated $F(\varphi)$ also ensures the boundness of $\varphi$ in the Cahn-Hilliard equation. Therefore, it
is a common practice to consider the Cahn-Hilliard equation with the truncated double-well potential $F(\varphi)$ \cite{shen_phase-field_2010,shen_decoupled_2015}. It is clear that the second derivative of $F(\varphi)$ is continuous and bounded,
\[
L:=\max_{\varphi\in\mathbb{R}}\left|F^{\prime\prime}(\varphi)\right|=2.
\]
This property is of great help in handling the nonlinear double-well potential  by  using stabilization method. For the treatment of the general double-well potential, we give some comments in Remark \ref{rem:stabch}.

\subsection{Weak formulation and energy estimate}
Firstly, we introduce some useful notations and Sobolev spaces. Let $L^{2}(\Omega)$
be the space of square-integrable functions that equipped with the
inner product and norm: 
\[
(f,g):=\int_{\Omega}fg\,{\rm d}x,\quad\|f\|_{0}:=(f,f)^{1/2},\quad\forall f,g\in L^{2}(\Omega).
\]
Its subspace with mean zero over $\Omega$ is written as $L_{0}^{2}(\Omega)$. Let $H^{1}(\Omega),\boldsymbol{H}({\rm div},\Omega)$ be the subspaces of $L^{2}(\Omega)$ with square integrable gradients and square integrable divergences respectively. The equipped norm in $\boldsymbol{H}({\rm div},\Omega)$ is defined by $$\left\Vert \boldsymbol{J}\right\Vert _{{\rm div},\Omega}=\left(\|\boldsymbol{J}\|_{0,\Omega}^{2}+\|{\rm div}\boldsymbol{J}\|_{0,\Omega}^{2}\right)^{1/2}.$$  
Their subspaces with vanishing traces and vanishing normal traces on $\Sigma$ are denoted by $H_{0}^{1}(\Omega),\boldsymbol{H}_{0}(\mathrm{div},\Omega)$.
We refer to \cite{monk_finite_2003} for their definitions and inner products. 

For convenience, we introduce some notations for function spaces
\begin{align*}
\boldsymbol{V} & \coloneqq\boldsymbol{H}_{0}^{1}(\Omega),\quad X\coloneqq H^{1}(\Omega),\quad Q=L_{0}^{2}(\Omega),\\
\boldsymbol{D} & \coloneqq\boldsymbol{H}_{0}(\mathrm{div},\Omega),\quad S=L_{0}^{2}(\Omega).
\end{align*}

To derive the weak formulation of (\ref{model:chimhd}), we
define a trilinear form:
\[
\begin{array}{l}
O(\boldsymbol{w},\boldsymbol{u},\boldsymbol{v})=\frac{1}{2}\left(\boldsymbol{w}\nabla\boldsymbol{u},\boldsymbol{v}\right)-\frac{1}{2}\left(\boldsymbol{w}\nabla\boldsymbol{v},\boldsymbol{u}\right),\end{array}
\]
for any $\boldsymbol{u},\boldsymbol{v},\boldsymbol{w}\in\boldsymbol{V}$.

Armed with the above notation, a weak formulation of the system (\ref{model:chimhd})
amounts to find $(\boldsymbol{u},p,\boldsymbol{J},\phi,\varphi,\mu)\in$ $\boldsymbol{V}\times Q\times\boldsymbol{D}\times S\times X\times X$
such that
\begin{equation}
\begin{aligned}\left\langle \boldsymbol{u}_{t},\boldsymbol{v}\right\rangle +2(\eta(\varphi)D(\boldsymbol{u}),D(\boldsymbol{v}))+\mathcal{O}(\boldsymbol{u},\boldsymbol{u},\boldsymbol{v})-(p,{\rm div}\boldsymbol{v})+(\varphi\nabla\mu,\boldsymbol{v})-\left(\boldsymbol{J}\times\boldsymbol{B},\boldsymbol{v}\right) & =\boldsymbol{0},\\
({\rm div}\boldsymbol{u},q) & =0,\\
\left(\sigma(\varphi)^{-1}\boldsymbol{J},\boldsymbol{K}\right)-\left(\boldsymbol{u}\times\boldsymbol{B},\boldsymbol{K}\right)-(\phi,{\rm div}\boldsymbol{K}) & =\boldsymbol{0},\\
({\rm div}\boldsymbol{J},\chi) & =0,\\
\left\langle \varphi_{t},\psi\right\rangle -(\varphi\boldsymbol{u},\nabla\psi)+M(\nabla\mu,\nabla\psi) & =0,\\
\lambda\varepsilon(\nabla\varphi,\nabla\chi)+\frac{\lambda}{\varepsilon}(f(\varphi),\chi)-(\mu,\chi) & =0,
\end{aligned}
\label{eq:weak}
\end{equation}

\noindent for all $(\boldsymbol{v},q,\boldsymbol{K},\theta,\psi,\chi)\in\boldsymbol{V}\times Q\times\boldsymbol{D}\times S\times X\times X$.

Now, we are in a position to establish the energy estimate for the
CHIMHD system.
\begin{thm}
\label{thm:weakengy}Let $(\boldsymbol{u},p,\boldsymbol{J},\phi,\varphi,\mu)$
be the solution of (\ref{eq:weak}). The following energy dissipation law holds
\[
\mathrm{E}(t)+\int_{0}^{t}\mathrm{P}(s)\mathrm{d}s=\mathrm{E}(0)
\]
where 
\begin{align*}
\mathrm{E}(t) & :=\int_{\Omega}\left(\frac{1}{2}|\boldsymbol{u}|^{2}+\frac{\lambda\epsilon}{2}|\nabla\varphi|^{2}+\frac{\lambda}{\epsilon}F(\varphi)\right){\rm d}x,\\
\mathrm{P}(t) & :=M\|\nabla\mu\|_{0}^{2}+2\|\sqrt{\eta(\varphi)}D(\boldsymbol{u})\|_{0}^{2}+\left\Vert \sqrt{\sigma(\varphi)^{-1}}\boldsymbol{J}\right\Vert _{0}^{2}.
\end{align*}
\end{thm}

\begin{proof}
Taking $(\boldsymbol{v},q,\boldsymbol{K},\theta,\psi,\chi)=(\boldsymbol{u},p,\boldsymbol{J},\phi,\mu,\varphi_{t})$ in (\ref{eq:weak}), we have
\[
\begin{array}{l}
\left(\boldsymbol{u}_{t},\boldsymbol{u}\right)+2\|\sqrt{\eta(\varphi)}D(\boldsymbol{u})\|_{0}^{2}-(\boldsymbol{J}\times\boldsymbol{B},\boldsymbol{u})+(\varphi\nabla\mu,\boldsymbol{u})=0,\\
\left(\sigma(\varphi)^{-1}\boldsymbol{J},\boldsymbol{J}\right)-(\boldsymbol{u}\times\boldsymbol{B},\boldsymbol{J})=0,\\
\left(\varphi_{t},\mu\right)-(\varphi\boldsymbol{u},\nabla\mu)+M\|\nabla\mu\|_{0}^{2}=0,\\
\lambda\varepsilon(\nabla\varphi,\nabla\varphi_{t})+\frac{\lambda}{\varepsilon}(f(\varphi),\varphi_{t})-(\mu,\varphi_{t})=0.
\end{array}
\]
Combining the these equalities, we obtain
\[
\left(\boldsymbol{u}_{t},\boldsymbol{u}\right)+\lambda\varepsilon(\nabla\varphi,\nabla\varphi_{t})+\frac{\lambda}{\varepsilon}\left(f(\varphi),\varphi_{t}\right)+2\|\sqrt{\eta(\varphi)}D(\boldsymbol{u})\|_{0}^{2}+\left\Vert \sqrt{\sigma(\varphi)^{-1}}\boldsymbol{J}\right\Vert _{0}^{2}+M\|\nabla\mu\|_{0}^{2}=0.
\]
 Integrating both sides over $(0,t)$ yields the theorem.
\end{proof}

The energy law describes the evolution of the total energy caused
by energy conversion. Since the induced magnetic field can be neglected
and the electric field is considered to be quasi-static, the total
energy $\mathrm{E}$ only consists of the fluid kinetic energy $\frac{1}{2}\|u\|_{0}^{2}$ and the Cahn-Hilliard free energy. The dissipation of $\mathrm{E}$ stems from the friction losses $\|\sqrt{\eta(\varphi)}D(\boldsymbol{u})\|_{0}^{2}$,
the Ohmic losses $\left\Vert\sqrt{\sigma(\varphi)^{-1}}\boldsymbol{J}\right\Vert _{0}^{2}$ and the diffusion transport term $M\|\nabla\mu\|_{0}^{2}$.

\section{Decoupled energy stable finite element method}
In this section, we propose a decoupled, energy stable, mixed finite element scheme for continuous problem (\ref{eq:weak}).

Let $\mathcal{T}_{h}$ be a shape-regular simplex subdivision of $\Omega$. As usual, we introduce the local mesh size $h_{K}=\mathrm{diam}\left(K\right)$
and global mesh size $h:=\underset{K\in\mathcal{T}_{h}}{\max}h_{K}$. Here we choose conforming finite element space pairs $\left(\boldsymbol{V}_{h},Q_{h}\right)\subset\left(\boldsymbol{V},Q\right)$
to discretize velocity \textbf{$\boldsymbol{u}$} and pressure $p$,
$\left(\boldsymbol{D}_{h},S_{h}\right)\subset\left(\boldsymbol{D},S\right)$
to approximate current density $\boldsymbol{J}$ and electric potential
$\varphi$, and $\left(X_{h},X_{h}\right)$ to discretize the 
phase field function $\phi$ and chemical potential $\mu$. In addition, we assume these spaces satisfy the following inf-sup conditions.
\begin{prop}[inf-sup condition]
The finite element pairs $\left(\boldsymbol{V}_{h},Q_{h}\right)$,
$\left(\boldsymbol{D}_{h},S_{h}\right)$ and $\left(X_{h},X_{h}\right)$ satisfy the following uniform inf-sup conditions:
\begin{align}
\inf_{q_{h}\in Q_{h}}\sup_{\boldsymbol{v}_{h}\in\boldsymbol{V}_{h}}\frac{(q_{h},{\rm div}\boldsymbol{v}_{h})}{\left\Vert \nabla\boldsymbol{v}_{h}\right\Vert _{0}\left\Vert q_{h}\right\Vert _{0}} & \geq\beta_{s},\\
\inf_{\theta_{h}\in S_{h}}\sup_{\boldsymbol{K}_{h}\in\boldsymbol{D}_{h}}\frac{(\psi_{h},{\rm div}\boldsymbol{K}_{h})}{\left\Vert \boldsymbol{K}_{h}\right\Vert _{{\rm _{div}}}\left\Vert \theta_{h}\right\Vert _{0}} & \geq\beta_{m},\\
\inf_{\chi_{h}\in X_{h}}\sup_{\psi_{h}\in X_{h}}\frac{(\nabla\chi_{h},\nabla\psi_{h})}{\left\Vert \chi_{h}\right\Vert _{1}\left\Vert \psi_{h}\right\Vert _{1}} & \geq\beta_{c},
\end{align}
where $\beta_{s}$, $\beta_{m}$ and \textbf{$\beta_{s}$} only depend
on $\Omega$.
\end{prop}

With these discrete spaces, the semi-discretization formulation of the system (\ref{eq:weak})
is to find $(\boldsymbol{u}_{h},p_{h},\boldsymbol{J}_{h},\phi_{h},\varphi_{h},\mu_{h})\in\boldsymbol{V}_{h}\times Q_{h}\times\boldsymbol{D}_{h}\times S_{h}\times X_{h}\times X_{h}$
such that

\begin{equation}
\begin{aligned}\left((\boldsymbol{u}_{h})_t,\boldsymbol{v}_{h}\right) +2(\eta(\varphi_{h})D(\boldsymbol{u}_{h}),D(\boldsymbol{v}_{h}))+\mathcal{O}(\boldsymbol{u}_{h},\boldsymbol{u}_{h},\boldsymbol{v}_{h})-(p_{h},{\rm div}\boldsymbol{v}_{h})&\\
+(\varphi_{h}\nabla\mu_{h},\boldsymbol{v}_{h})-\left(\boldsymbol{J}_{h}\times\boldsymbol{B}_{h},\boldsymbol{v}_{h}\right) & =\boldsymbol{0},\\
({\rm div}\boldsymbol{u}_{h},q_{h}) & =0,\\
\left(\sigma(\varphi_{h})^{-1}\boldsymbol{J}_{h},\boldsymbol{K}_{h}\right)-\left(\boldsymbol{u}_{h}\times\boldsymbol{B}_{h},\boldsymbol{K}_{h}\right)-(\phi_{h},{\rm div}\boldsymbol{K}_{h}) & =\boldsymbol{0},\\
({\rm div}\boldsymbol{J}_{h},\chi_{h}) & =0,\\
\left((\varphi_{h})_t,\psi_{h}\right) -(\varphi_{h}\boldsymbol{u}_{h},\nabla\psi_{h})+M(\nabla\mu_{h},\nabla\psi_{h}) & =0,\\
\lambda\varepsilon(\nabla\varphi_{h},\nabla\chi_{h})+\frac{\lambda}{\varepsilon}(f(\varphi_{h}),\chi_{h})-(\mu_{h},\chi_{h}) & =0,
\end{aligned}
\label{eq:weakh}
\end{equation}

\noindent for all $(\boldsymbol{v}_{h},q_{h},\boldsymbol{K}_{h},\theta_{h},\psi_{h},\chi_{h})\in\boldsymbol{V}_{h}\times Q_{h}\times\boldsymbol{D}_{h}\times S_{h}\times X_{h}\times X_{h}$.

With similar arguments in Theorem \ref{thm:weakengy}, one can easily get the energy stability of the semi-discrete scheme, thus the details are omitted here.
\begin{thm}
\label{thm:weakhengy}Let $(\boldsymbol{u},p,\boldsymbol{J},\phi,\varphi,\mu)$
be the solution of (\ref{eq:weak}). The energy dissipation law holds
\[
\mathrm{E}(t)+\int_{0}^{t}\mathrm{P}(s)\mathrm{d}s=\mathrm{E}(0)
\]
where 
\begin{align*}
\mathrm{E}(t) & :=\int_{\Omega}\left(\frac{1}{2}|\boldsymbol{u}_{h}|^{2}+\frac{\lambda\epsilon}{2}|\nabla\varphi_{h}|^{2}+\frac{\lambda}{\epsilon}F(\varphi_{h})\right){\rm d}x,\\
\mathrm{P}(t) & :=M\|\nabla\mu_{h}\|_{0}^{2}+2\|\sqrt{\eta(\varphi_{h})}D(\boldsymbol{u}_{h})\|_{0}^{2}+\left\Vert \sqrt{\sigma(\varphi_{h})^{-1}}\boldsymbol{J}_{h}\right\Vert _{0}^{2}.
\end{align*}
\end{thm}

Let $\left\{ t_{n}=n\tau:\,\, n=0,1,\cdots,N\right\} ,\tau=T/N,$
be an equidistant partition of the time interval $[0,T].$ For any
time dependent function $\omega\left(x,t\right)$, the full discrete
approximation to $\omega\left(t_{n}\right)$ will be denoted by $\omega_{h}^{n}$. A fully discrete mixed finite element scheme for problem (\ref{eq:weak}) reads as follows: 

Given the initial guess $\boldsymbol{u}^{0}$ and $\phi^{0}$, we
compute $$\left(\boldsymbol{u}_{h}^{n+1},p_{h}^{n+1},\boldsymbol{J}_{h}^{n+1},\varphi_{h}^{n+1},\phi_{h}^{n+1},\mu_{h}^{n+1}\right),\,n=0,1,\cdots,N-1,$$ by the following three steps.\\
\textbf{Step 1:} Find $\left(\varphi_{h}^{n+1},\mu_{h}^{n+1}\right)\in X_{h}\times X_{h}$
such that
\begin{equation}
\begin{cases}
\left(\delta_{t}\varphi_{h}^{n+1},\psi_{h}\right)+M\left(\nabla\mu_{h}^{n+1},\nabla\psi_{h}\right)+\tau\left(\varphi_{h}^{n}\nabla\mu_{h}^{n+1},\varphi_{h}^{n}\nabla\psi_{h}\right) & =\left(\varphi_{h}^{n}\boldsymbol{u}_{h}^{n},\nabla\psi_{h}\right),\\
\lambda\varepsilon\left(\nabla\varphi_{h}^{n+1},\nabla\chi_{h}\right)+\left(\frac{\lambda}{\epsilon}(\varphi_{h}^{n+1}-\varphi_{h}^{n}),\chi_{h}\right)\\
+\left(\frac{\lambda}{\varepsilon}f(\varphi^{n}),\chi_{h}\right)-\left(\mu_{h}^{n+1},\chi_{h}\right) & =0,
\end{cases}\label{eq:CHt}
\end{equation}
for all $(\psi_{h},\chi_{h})\in X_{h}\times X_{h}$.\\
\textbf{Step 2:} Find $\left(\boldsymbol{J}_{h}^{n+1},\phi_{h}^{n+1}\right)\in\boldsymbol{D}_{h}\times S_{h}$
such that
\begin{equation}
\begin{cases}
\left(\sigma\left(\varphi_{h}^{n+1}\right)^{-1}\boldsymbol{J}_{h}^{n+1},\boldsymbol{K}_{h}\right)+\tau\left(\boldsymbol{J}_{h}^{n+1}\times\boldsymbol{B},\boldsymbol{K}_{h}^{n+1}\times\boldsymbol{B}\right)\\
-\left(\phi_{h}^{n+1},\mathrm{div}\boldsymbol{K}_{h}\right)-\tau\left(\varphi_{h}^{n+1}\nabla\mu_{h}^{n+1}\times\boldsymbol{B},\boldsymbol{K}_{h}^{n+1}\right) & =\left(\boldsymbol{u}_{h}^{n}\times\boldsymbol{B},\boldsymbol{K}_{h}^{n+1}\right),\\
\left(\mathrm{div}\boldsymbol{J}_{h}^{n+1},\theta_{h}\right) & =0,
\end{cases}\label{eq:Jphit}
\end{equation}
for all $(\boldsymbol{K}_{h},\theta_{h})\in\boldsymbol{D}_{h}\times S_{h}$.\\
\textbf{Step 3:} Find $\left(\boldsymbol{u}^{n+1},p^{n+1}\right)\in\boldsymbol{V}_{h}\times Q_{h}$
such that
\begin{equation}
\begin{cases}
(\delta_{t}\boldsymbol{u}_{h}^{n+1},\boldsymbol{v}_{h})+O(\boldsymbol{u}_{h}^{n},\boldsymbol{u}_{h}^{n+1},\boldsymbol{v}_{h})-2\left(\eta\left(\varphi_h^{n+1}\right)D\left(\boldsymbol{u}_{h}^{n+1}\right),D\left(\boldsymbol{v}_{h}^{n+1}\right)\right)\\
\qquad-\left(p_{h}^{n+1},{\rm div}\boldsymbol{v}_{h}\right)  +\left(\varphi_{h}^{n+1}\nabla\mu_{h}^{n+1},\boldsymbol{v}_{h}\right)+\left(\boldsymbol{J}_{h}^{n+1}\times\boldsymbol{B},\boldsymbol{v}_{h}\right)&=\boldsymbol{0},\\
\left(\mathrm{div}\boldsymbol{u}^{n+1},q_{h}\right) & =0,
\end{cases}\label{eq:NSt}
\end{equation}
for all $(\boldsymbol{v}_{h},q_{h})\in\boldsymbol{V}_{h}\times Q_{h}$.

Before we get into the discussion on the properties of this scheme, several remarks about the scheme are given in order.

\begin{rem}
	To decouple the nonlinear coupled multiphysics system, we introduce two first -order stabilization terms. The first stabilization term $\tau(\varphi^n_{h}\nabla \mu^n_{h},\varphi^n_{h}\nabla \psi_{h})$ in step 1 is to decouple Cahn--Hilliard equations and Navier--Stokes equations \cite{shen_numerical_2010}. The second stabilization term $\tau\left(\boldsymbol{J}_{h}^{n+1}\times\boldsymbol{B},\boldsymbol{K}_{h}\times\boldsymbol{B}\right)$ in step 2 is to decouple Poisson equation and Navier--Stokes equations \cite{Zhang2020}. These extra two stabilization terms are vital to  keep the couping term explicitly while preserving the energy stability, see the proof of Theorem \ref{thm:weakhengy}.
\end{rem}

\begin{rem}\label{rem:stabch}
In step 1, we employ the stabilized method \cite{shen_numerical_2010} to treat the double-well potential $F(\varphi)$ explicitly without suffering from any time step constraint. Note that this stabilizing term $\frac{\lambda}{\epsilon}(\varphi^{n+1}_{h}-\varphi^n_{h})$ introduces an extra consistent error of order $\tau$, which is the same order as the overall truncation error of the scheme. Note that there are many other efficient methods on constructing energy stable schemes for the Cahn-Hilliard equation, such as, convex-splitting method   \cite{eyre_unconditionally_nodate}, invariant energy quadratization method \cite{Yang2017}, and scalar auxiliary variable method \cite{Shen2019}. Here we adopt the stabilized explicit method only for simplicity, and it is easy to extend the scheme to the methods mentioned above. 
\end{rem}

\begin{rem}
In this paper, we only focus on the decoupling of multiphysics problem rather the decoupling of all variable in each physical problem. The velocity and pressure in step 3 can be further decoupled by using the first order pressure correction scheme \cite{shen_decoupled_2015}, and we leave it to the interested readers.
\end{rem}

It is clear that the scheme given by (\ref{eq:CHt})-(\ref{eq:NSt}) is a decoupled, linear scheme. Next we want to show that the scheme is mass-conservative and charge-conservative.
\begin{prop}
Let $\left(\boldsymbol{u}_{h}^{m},p_{h}^{m},\boldsymbol{J}_{h}^{m},\phi_{h}^{m},\varphi_{h}^{m},\mu_{h}^{m}\right)$
solve (\ref{eq:CHt})-(\ref{eq:NSt}) for any $1\le m\le N$, then
the scheme is mass-conservative and charge-conservative, namely, 
\[
\int_{\Omega}\varphi_{h}^{m}\mathrm{~d}x=\int_{\Omega}\varphi_{h}^{0}\mathrm{~d}x,\quad{\rm div}\boldsymbol{J}_{h}^{m}=0.
\]
\end{prop}
\begin{proof}
Letting $\psi_{h}=1$ in the first equation of (\ref{eq:CHt}), we
have mass conservation $\int_{\Omega}\varphi_{h}^{m}\mathrm{~d}x=\int_{\Omega}\varphi_{h}^{0}\mathrm{~d}x$.
Then, we note that for all $\theta_{h}\in S_{h}$, there holds
\[
\left(\theta_{h},{\rm div}\boldsymbol{J}_{h}^m\right)=0,
\]
and ${\rm div}\boldsymbol{J}_{h}^m\in S_{h}$. Taking $\theta_{h}={\rm div}\boldsymbol{J}_{h}^m$,
we obtain ${\rm div}\boldsymbol{J}_{h}^m=0$.
\end{proof}

Now, we are in a position to prove the unconditional energy stability of the decoupled scheme as follows, which is analogous to that of the original problem in Theorem \ref{thm:weakengy}.
\begin{thm}
\label{thm:SeEnergyLaws} The decoupled scheme is unconditionally
energy stable in the sense that the following energy estimate
\begin{equation}
\delta_{t}\mathrm{E}^{n+1}+\mathrm{P}^{n}\le0\quad\forall n\ge0,\label{eq:SemiEn}
\end{equation}
 holds, where 
\begin{align*}
\mathrm{E}^{n+1} & :=\frac{1}{2}\|\boldsymbol{u}_{h}^{n+1}\|_{0}^{2}+\frac{\lambda\epsilon}{2}\|\nabla\varphi_{h}^{n+1}\|_{0}^{2}+\frac{\lambda}{\epsilon}\left(F\left(\varphi_{h}^{n+1}\right),1\right),\\
\mathrm{P}^{n} & :=M\|\nabla\mu_{h}^{n+1}\|_{0}^{2}+2\|\sqrt{\eta(\varphi_{h}^{n+1})}D(\boldsymbol{u}_{h}^{n+1})\|_{0}^{2}+\left\Vert \sqrt{\sigma(\varphi_{h}^{n+1})^{-1}}\boldsymbol{J}_{h}^{n+1}\right\Vert _{0}^{2}.
\end{align*}
\end{thm}
\begin{proof}
Letting $\left(\psi_{h},\chi_{h}\right)=\left(\mu_{h}^{n+1},\delta_{t}\varphi_{h}^{n+1}\right)$
in (\ref{eq:CHt}), we have
\begin{align}
\left(\delta_{t}\varphi_{h}^{n+1},\mu_{h}^{n+1}\right)+\tau\|\varphi_{h}^{n}\nabla\mu_{h}^{n+1}\|^{2}+M\|\nabla\mu_{h}^{n+1}\|^{2} & =(\varphi_{h}^{n}\boldsymbol{u}_{h}^{n},\nabla\mu_{h}^{n+1}),\label{eq:SemiCHen1}\\
\lambda\varepsilon(\nabla\varphi_{h}^{n+1},\nabla\delta_{t}\varphi_{h}^{n+1})+\frac{\lambda}{\varepsilon}(f(\varphi_{h}^{n}),\delta_{t}\varphi_{h}^{n+1})+\frac{1}{\epsilon\tau}\|\varphi_{h}^{n+1}-\varphi_{h}^{n}\|^{2} & =(\mu_{h}^{n+1},\delta_{t}\varphi_{h}^{n+1}).\label{eq:SemiCHen2}
\end{align}

Next, setting $\left(\boldsymbol{K}_{h},\theta_{h}\right)=\left(\boldsymbol{J}_{h}^{n+1},\phi_{h}^{n+1}\right)$,
we get
\begin{equation}
\left\Vert \sqrt{\sigma(\varphi_{h}^{n+1})^{-1}}\boldsymbol{J}_{h}^{n+1}\right\Vert _{0}^{2}+\tau\left\Vert \boldsymbol{J}_{h}^{n+1}\times\boldsymbol{B}\right\Vert _{0,\Omega}^{2}=-(\boldsymbol{u}_{h}^{n},\boldsymbol{J}_{h}^{n+1}\times\boldsymbol{B})+\tau(\varphi_{h}^{n}\nabla\mu_{h}^{n+1},\boldsymbol{J}_{h}^{n+1}\times\boldsymbol{B}).\label{eq:SemiJen}
\end{equation}

Then, taking $\left(\boldsymbol{v}_{h},q_{h}\right)=\left(\boldsymbol{u}_{h}^{n+1},p_{h}^{n+1}\right)$,
and using the identity 
\[
\left(\delta_{t}\boldsymbol{u}_{h}^{n+1},\boldsymbol{u}_{h}^{n+1}\right)=\frac{1}{2\tau}\left(\left\Vert \boldsymbol{u}_{h}^{n+1}\right\Vert _{0,\Omega}^{2}-\left\Vert \boldsymbol{u}_{h}^{n}\right\Vert _{0,\Omega}^{2}+\left\Vert \boldsymbol{u}_{h}^{n+1}-\boldsymbol{u}_{h}^{n}\right\Vert _{0,\Omega}^{2}\right),
\]
\noindent it yields
\begin{align}
&\frac{1}{2\tau}\left(\left\Vert \boldsymbol{u}_{h}^{n+1}\right\Vert _{0,\Omega}^{2}-\left\Vert \boldsymbol{u}_{h}^{n}\right\Vert _{0,\Omega}^{2}+\left\Vert \boldsymbol{u}_{h}^{n+1}-\boldsymbol{u}_{h}^{n}\right\Vert _{0,\Omega}^{2}\right)+2\|\sqrt{\eta(\varphi_{h}^{n+1})}D(\boldsymbol{u}_{h}^{n+1})\|_{0}^{2}\nonumber\\
&\qquad=\left(\boldsymbol{J}_{h}^{n+1}\times\boldsymbol{B}_{h},\boldsymbol{u}_{h}^{n+1}\right)-\left(\varphi_{h}^{n}\nabla\mu_{h}^{n+1},\boldsymbol{u}_{h}^{n+1}\right).\label{eq:Semiuen}
\end{align}

\noindent By making the summations of (\ref{eq:SemiCHen1})-(\ref{eq:Semiuen}),
we obtain
\begin{align}
&\frac{1}{2\tau}\left(\left\Vert \boldsymbol{u}_{h}^{n+1}\right\Vert _{0,\Omega}^{2}-\left\Vert \boldsymbol{u}_{h}^{n}\right\Vert _{0,\Omega}^{2}+\left\Vert \boldsymbol{u}_{h}^{n+1}-\boldsymbol{u}_{h}^{n}\right\Vert _{0,\Omega}^{2}\right)+2\|\sqrt{\eta(\varphi_{h}^{n+1})}D(\boldsymbol{u}_{h}^{n+1})\|_{0}^{2}\nonumber \\
&+\left\Vert \sqrt{\sigma(\varphi_{h}^{n+1})^{-1}}\boldsymbol{J}_{h}^{n+1}\right\Vert _{0}^{2}+\tau\left\Vert \boldsymbol{J}_{h}^{n+1}\times\boldsymbol{B}\right\Vert _{0,\Omega}^{2}+\tau\|\varphi_{h}^{n}\nabla\mu_{h}^{n+1}\|^{2}+M\|\nabla\mu_{h}^{n+1}\|^{2}\nonumber \\
&=\left(\boldsymbol{J}_{h}^{n+1}\times\boldsymbol{B}_{h}-\varphi_{h}^{n}\nabla\mu_{h}^{n+1},\boldsymbol{u}_{h}^{n+1}-\boldsymbol{u}_{h}^{n}\right)\nonumber \\
&\quad+\tau(\varphi_{h}^{n}\nabla\mu_{h}^{n+1},\boldsymbol{J}_{h}^{n+1}\times\boldsymbol{B}_{h})\nonumber \\
&\quad-\frac{\lambda}{\varepsilon}(f(\varphi_{h}^{n}),\delta_{t}\varphi_{h}^{n+1})-\frac{1}{\epsilon\tau}\|\varphi_{h}^{n+1}-\varphi_{h}^{n}\|^{2}.\label{eq:Semisum}
\end{align}

Using Young inequality, we derive the right hand side of the first
term of (\ref{eq:Semisum}) has the following estimate,
\begin{align}
&\left(\boldsymbol{J}_{h}^{n+1}\times\boldsymbol{B}_{h}-\varphi_{h}^{n}\nabla\mu_{h}^{n+1},\boldsymbol{u}_{h}^{n+1}-\boldsymbol{u}_{h}^{n}\right)\nonumber \\
& \le\left\Vert \boldsymbol{J}_{h}^{n+1}\times\boldsymbol{B}_{h}-\varphi_{h}^{n}\nabla\mu_{h}^{n+1}\right\Vert _{0,\Omega}\left\Vert \boldsymbol{u}_{h}^{n+1}-\boldsymbol{u}_{h}^{n}\right\Vert _{0,\Omega}\nonumber \\
 & \le\frac{\tau}{2}\left\Vert \boldsymbol{J}_{h}^{n+1}\times\boldsymbol{B}_{h}-\varphi_{h}^{n}\nabla\mu_{h}^{n+1}\right\Vert _{0,\Omega}^{2}+\frac{1}{2\tau}\left\Vert \boldsymbol{u}_{h}^{n+1}-\boldsymbol{u}_{h}^{n}\right\Vert _{0,\Omega}^{2}\nonumber\\
 & =\frac{\tau}{2}\left\Vert \boldsymbol{J}_{h}^{n+1}\times\boldsymbol{B}_{h}\right\Vert _{0,\Omega}^{2}-\tau(\varphi_{h}^{n}\nabla\mu_{h}^{n+1},\boldsymbol{J}_{h}^{n+1}\times\boldsymbol{B}_{h})\nonumber \\
 & \quad+\frac{\tau}{2}\left\Vert \varphi_{h}^{n}\nabla\mu_{h}^{n+1}\right\Vert _{0,\Omega}^{2}+\frac{1}{2\tau}\left\Vert \boldsymbol{u}_{h}^{n+1}-\boldsymbol{u}_{h}^{n}\right\Vert _{0,\Omega}^{2}.\label{eq:Semirhs1}
\end{align}
For the last two term in (\ref{eq:Semisum}), using the Taylor expansion,
\[
F\left(\varphi_{h}^{n+1}\right)-F\left(\varphi_{h}^{n}\right)=\left(f\left(\varphi_{h}^{n}\right),\varphi_{h}^{n+1}-\varphi_{h}^{n}\right)+\frac{f^{\prime}\left(\xi_{h}^{n}\right)}{2}\|\varphi_{h}^{n+1}-\varphi_{h}^{n}\|^{2},
\]
\noindent then we have
\begin{align}
&-\frac{\lambda}{\varepsilon}(f(\varphi_{h}^{n}),\delta_{t}\varphi_{h}^{n+1})-\frac{\lambda }{\epsilon\tau}\|\varphi_{h}^{n+1}-\varphi_{h}^{n}\|^{2}\nonumber\\
&=-\frac{\lambda}{\epsilon\tau}\left(F\left(\varphi_{h}^{n+1}\right)-F\left(\varphi_{h}^{n}\right)\right)+\frac{\lambda}{\epsilon\tau}\left(\frac{f^{\prime}\left(\xi_{h}^{n}\right)}{2}-1\right)\|\varphi_{h}^{n+1}-\varphi_{h}^{n}\|^{2}.\label{eq:Semirhs2}
\end{align}

\noindent Plugging (\ref{eq:Semirhs1}) and (\ref{eq:Semirhs2}) into
(\ref{eq:Semisum}), and dropping the positive term $\frac{\tau}{2}\left\Vert \boldsymbol{J}_{h}^{n+1}\times\boldsymbol{B}\right\Vert _{0,\Omega}^{2}$
and $\frac{\tau}{2}\|\varphi_{h}^{n}\nabla\mu_{h}^{n+1}\|^{2}$, we
get the required estimate (\ref{eq:SemiEn}). The proof is thus complete.
\end{proof}

\section{Numerical Experiments}
In this section, we present a series of 2D numerical experiments to
illustrate the features of the proposed algorithms. The
finite element method is implemented on the finite element software
FreeFEM developed by \cite{hecht_new_2012}. For any integer $k\geq0,$
let $P_{k}(K)$ be the space of polynomials of degree $k$ on element
$K$, and denote $\boldsymbol{P}_{k}(K)=P_{k}(K)^{2}$. We employ the
Mini-element \cite{arnold_stable_1984} to approximate the velocity and pressure
\[
\boldsymbol{V}_{h}=\boldsymbol{P}_{1,h}^{b}\cap\boldsymbol{V},\quad Q_{h}=\left\{ q_{h}\in Q:\left.q_{h}\right|_{K}\in P_{1}(K),\,\forall K\in\mathcal{T}_{h}\right\}, 
\]
where $P_{1,h}^{b}=\left\{ v_{h}\in C^{0}(\Omega):\left.v_{h}\right|_{K}\in P_{1}(K)\oplus{\rm span}\{\hat{b}\},\,\forall K\subset T_{h}\right\} $,
$\hat{b}$ is a bubble function. We choose
the lowest-order Raviart-Thomas \cite{raviart_mixed_1977} element space
\[
\boldsymbol{D}_{h}=\left\{ \boldsymbol{K}_{h}\in\boldsymbol{D}:\left.\boldsymbol{K}_{h}\right|_{K}\in\boldsymbol{P}_{0}(K)+\boldsymbol{x}P_{0}(K),\,\forall K\in\mathcal{T}_{h}\right\} ,
\]
to approximate the current density, the discontinuous and piecewise constant finite element space to approximate the electric potential
\[
S_{h}=\left\{ \psi_{h}\in S:\left.\psi_{h}\right|_{K}\in P_{0}(K),\,\forall K\in\mathcal{T}_{h}\right\} .
\]
 The phase field $\varphi$ and chemical potential $\mu$ are
discretized by first order Lagrange finite element space $\left(X_{h},X_{h}\right)$, where
\[
X_{h}=\left\{ \chi_{h}\in L_{0}^{2}:\left.\chi_{h}\right|_{K}\in P_{1}(K),\,\forall K\in\mathcal{T}_{h}\right\}.
\]

\begin{example}[Convergence and accuracy]
 The first example is used to verify the convergence rates in both
time and space. The computational domain is set as $\Omega=\left(0,1\right){}^{2}$, and the external magnetic field is \textbf{$\boldsymbol{B}=\left(0,0,1\right)^{{\rm T}}$}.
The physical parameters are given by $R_{e}=\kappa=1$ with terminal
time $T=1$. The right-hand side functions, initial conditions and Dirichlet boundary conditions are chosen such that the given solutions satisfy the system. 
\end{example}

Let the approximation errors at the final time $t=T$
be denoted by
\[
e_{\omega}=\omega(T)-\omega_{h}^{N}\quad\omega\in\left\{ \boldsymbol{u},p,\boldsymbol{J},\phi,\varphi,\mu\right\} .
\]

First, we test the temporal convergence orders. The analytic solutions
are chosen as
\begin{align*}
\boldsymbol{u} & =\left(y\exp\left(-t\right),x\cos\left(t\right)\right),\quad p=\sin\left(t\right),\\
\boldsymbol{J} & =\left(\sin\left(t\right),\cos\left(t\right)\right),\quad\phi=1,\\
\varphi & =(x+y)\exp\left(-t\right),\quad\mu=x\cos t.
\end{align*}
Note that the exact solutions are linear or constant in space, the main error comes from the discretization of the time variable. We
fix a mesh size with $h=1/10$ and test the convergence rate with respect
to the time step. Then the errors and orders are displayed in Tables \ref{tab:Tup}-\ref{tab:Tcmu}.
From these tables, we observe that the errors of all variable decrease as the mesh is refined, with convergence order of $O(\tau)$, which accords with our theoretical analysis completely. 

\begin{table}
\caption{Time convergence rates of the scheme for $\left(\boldsymbol{u},p\right)$
\label{tab:Tup}}
\begin{centering}
\begin{tabular}{cccc}
\toprule 
$\tau$ & $\|e_{\boldsymbol{u}}\|_{0}$ & $\|\nabla e_{\boldsymbol{u}}\|_{0}$ & $\|e_{p}\|_{0}$\tabularnewline
\midrule 
0.2 & 1.7227e-4(\textemdash ) & 1.3384e-3(\textemdash ) & 3.4903e-2(\textemdash )\tabularnewline
\midrule 
0.1 & 7.5918e-05(1.18) & 5.9037e-4(1.18) & 1.6788e-2(1.06)\tabularnewline
\midrule 
0.05 & 3.5535e-05(1.10) & 2.7649e-4(1.09) & 8.1861e-3(1.04)\tabularnewline
\midrule 
0.025 & 1.7206e-05(1.05) & 1.3391e-4(1.05) & 4.0362e-3(1.02)\tabularnewline
\midrule 
0.0125 & 8.4693e-06(1.02) & 6.5926e-05(1.02) & 2.0033e-3(1.01)\tabularnewline
\midrule 
0.00625 & 4.2022e-06(1.01) & 3.2712e-05(1.01) & 9.9790e-4(1.01)\tabularnewline
\midrule 
\end{tabular}
\par\end{centering}
\caption{Time convergence rates of the scheme for $\left(\boldsymbol{J},\phi\right)$\label{tab:Tjphi}}

\begin{centering}
\begin{tabular}{cccc}
\toprule 
$\tau$ & $\|e_{\boldsymbol{J}}\|_{\text{div }}$ & $\|e_{\phi}\|_{0}$ & $\|\nabla\cdot\boldsymbol{J}_{h}^{N}\|_{0}$\tabularnewline
\midrule 
0.2 & 6.9951e-3(\textemdash ) & 3.6822e-2(\textemdash ) & 3.68218e-12\tabularnewline
\midrule 
0.1 & 3.6099e-4(0.95) & 1.9026e-2(0.95) & 1.90247e-12\tabularnewline
\midrule 
0.05 & 1.8227e-3(0.99) & 9.6627e-3(0.98) & 9.66331e-13\tabularnewline
\midrule 
0.025 & 9.1439e-4(1.00) & 4.8671e-3(0.99) & 4.86857e-13\tabularnewline
\midrule 
0.0125 & 4.5781e-4(1.00) & 2.4422e-3(0.99) & 2.44362e-13\tabularnewline
\midrule 
0.00625 & 2.2904e-4(1.00) & 1.2233e-3(1.00) & 1.22111e-13\tabularnewline
\bottomrule
\end{tabular}
\par\end{centering}
\caption{Time convergence rates of the scheme for $\left(\varphi,\mu\right)$\label{tab:Tcmu}}

\centering{}%
\begin{tabular}{ccccc}
\midrule 
$\tau$ & $\|e_{\varphi}\|_{0}$ & $\|\nabla e_{\varphi}\|_{0}$ & $\|e_{\mu}\|_{0}$ & $\|\nabla e_{\mu}\|_{0}$\tabularnewline
\midrule 
0.2 & 1.1298e-1(\textemdash ) & 1.8314e-2(\textemdash ) & 1.1941e-1 (\textemdash ) & 6.1937e-3(\textemdash )\tabularnewline
\midrule 
0.1 & 5.1171e-2(1.14) & 5.9434e-3(1.62) & 7.0687e-2(0.77) & 3.2132e-3(0.95)\tabularnewline
\midrule 
0.05 & 2.4156e-2(1.08) & 2.2442e-3(1.40) & 3.6871e-2(0.94) & 1.5619e-3(1.04)\tabularnewline
\midrule 
0.025 & 1.1710e-2(1.04) & 9.6111e-4(1.22) & 1.8674e-2(0.98) & 7.6263e-4(1.03)\tabularnewline
\midrule 
0.0125 & 5.7621e-3(1.02) & 4.4380e-4(1.11) & 9.3800e-3(0.99) & 3.7607e-4(1.02)\tabularnewline
\midrule 
0.00625 & 2.8576e-3(1.01) & 2.1321e-4(1.06) & 4.6987e-3(1.00) & 1.8666e-5(1.01)\tabularnewline
\bottomrule
\end{tabular}
\end{table}

Next, we aim to check the spatial approximation orders. To this end, we choose the exact solution 
\begin{align*}
\boldsymbol{u} & =\left(\sin\left(y\right)\exp\left(-t\right),x^{2}\cos\left(t\right)\right),\quad p=y\sin\left(t\right),\\
\boldsymbol{J} & =\left(y^{2}\sin\left(t\right),\sin\left(x\right)\cos\left(t\right)\right),\quad\phi=x\exp\left(-t\right),\\
\varphi & =\sin(x)\exp(-t),\quad\mu=\cos(y)\cos(t).
\end{align*}
With initial time step and mesh width $h_{0}=2\tau_{0}=1/2$, we simultaneously refine time and space size such that the relation
$h=2\tau$ holds. The corresponding convergent results
are demonstrated in Tables \ref{tab:Supjphi}-\ref{tab:phicmu} and a first order convergence of the proposed numerical scheme can be observed
asymptotically, which agrees with our expected results.

\begin{table}
\begin{centering}
\caption{Full discretization convergence rates of the scheme for $\left(\boldsymbol{u},p,\boldsymbol{J}\right)$\label{tab:Supjphi}}
\par\end{centering}
\begin{centering}
\begin{tabular}{ccccc}
\hline 
$\ensuremath{\left(\tau,h\right)}$ & $\|\nabla e_{\boldsymbol{u}}\|_{0}$ & $\|e_{p}\|_{0}$ & $\|e_{\boldsymbol{J}}\|_{\text{div }}$ & $\|\nabla\cdot\boldsymbol{J}_{h}^{N}\|_{0}$\tabularnewline
\hline 
$\ensuremath{\left(\tau_{0},h_{0}\right)}$ & 1.1685e-1(\textemdash ) & 5.4282e-2(\textemdash ) & 1.5322e-1(\textemdash ) & 1.02906e-11\tabularnewline
\hline 
$\ensuremath{\left(\tau_{0},h_{0}\right)}/2$ & 5.7591e-2(1.02) & 1.8775e-2(1.53) & 7.7659e-2(0.98) & 1.03095e-11\tabularnewline
\hline 
$\ensuremath{\left(\tau_{0},h_{0}\right)}/4$ & 2.8635e-2(1.01) & 7.4194e-3(1.34) & 3.9010e-2(0.99) & 1.04248e-11\tabularnewline
\hline 
$\ensuremath{\left(\tau_{0},h_{0}\right)}/8$ & 1.4287e-2(1.00) & 3.3507e-3(1.15) & 1.9536e-2(1.00) & 1.05131e-11\tabularnewline
\hline 
$\ensuremath{\left(\tau_{0},h_{0}\right)}/16$ & 7.1370e-3(1.00) & 1.6213e-3(1.05) & 9.7733e-3(1.00) & 1.05643e-11\tabularnewline
\hline 
$\ensuremath{\left(\tau_{0},h_{0}\right)}/32$ & 3.5671e-3(1.00) & 8.0384e-4(1.01) & 4.8877e-3(1.00) & 1.05915e-11\tabularnewline
\hline 
\end{tabular}
\par\end{centering}
\centering{}\caption{Full discretization convergence rates of the scheme for $\left(\phi,\varphi,\text{\ensuremath{\mu}}\right)$\label{tab:phicmu}}
\begin{tabular}{cccc}
\hline 
$\ensuremath{\left(\tau,h\right)}$ & $\|e_{\phi}\|_{0}$ & $\|\nabla e_{\varphi}\|_{0}$ & $\|\nabla e_{\mu}\|_{0}$\tabularnewline
\hline 
$\ensuremath{\left(\tau_{0},h_{0}\right)}$ & 4.5526e-2(\textemdash ) & 2.8424e-2(\textemdash ) & 6.3754e-2(\textemdash )\tabularnewline
\hline 
$\ensuremath{\left(\tau_{0},h_{0}\right)}/2$ & 2.2281e-2(1.00) & 1.4898e-2(0.93) & 3.2766e-2(0.96)\tabularnewline
\hline 
$\ensuremath{\left(\tau_{0},h_{0}\right)}/4$ & 1.1450e-2(0.99) & 7.5740e-3(0.98) & 1.6567e-2(0.98)\tabularnewline
\hline 
$\ensuremath{\left(\tau_{0},h_{0}\right)}/8$ & 5.7815e-3(0.99) & 3.8063e-3(0.99) & 8.3157e-3(0.99)\tabularnewline
\hline 
$\ensuremath{\left(\tau_{0},h_{0}\right)}/16$ & 2.900e-3(1.00) & 1.9060e-3(1.00) & 4.1632e-3(1.00)\tabularnewline
\hline 
$\ensuremath{\left(\tau_{0},h_{0}\right)}/32$ & 1.4520e-3(1.00) & 9.5336e-4(1.00) & 2.0825e-3(1.00)\tabularnewline
\hline 
\end{tabular}
\end{table}

Finally, we verify the exactly divergence-free property of the discrete current
density. From the last column of Tables \ref{tab:Tjphi} and \ref{tab:Supjphi},
the approximate solutions yields $\left\Vert \nabla\cdot\boldsymbol{J}_{h}\right\Vert _{0}$
in the order of $10^{-11}\sim10^{-13}$, which is almost divergence-free.
These tiny errors mainly result from the numerical integral errors
and rounding errors. 
\begin{example}[Shape relaxation]
In this example, we simulate the evolution of a square shaped bubble
and two kissing circular bubbles in the domain $\Omega=(0,1)^{2}$.
We set the external magnetic field as \textbf{$\boldsymbol{B}=\left(0,0,1\right)^{{\rm T}}$}
and the physical parameters
\[
\eta=\sigma=1,\quad\epsilon=0.01,\quad\gamma=M=0.1.
\]
\noindent The initial velocity is taken as zero vector. For a square shaped bubble, the initial profile of phase function $\varphi$ is chosen
to be
\[
\varphi_{0}=\tanh\left(\frac{|x+y-1|+|x-y|-0.4}{\sqrt{2}\varepsilon}\right).
\]

\noindent For two kissing circular bubbles, $\varphi_{0}$ is taken
as
\[
\varphi_{0}=1-\tanh\left(\frac{\left\Vert x-x_{o_{1}}\right\Vert -r_{1}}{\sqrt{2}\varepsilon}\right)-\tanh\left(\frac{\left\Vert x-x_{o_{2}}\right\Vert -r_{2}}{\sqrt{2}\varepsilon}\right)
\]

\noindent where$\left\Vert x-x_{o_{i}}\right\Vert $ is the Eulerian
distance between the points $x$ and $x_{o_{i}}$. The two points $x_{o_{1}}=\left(0.3,0.5\right)$
and $x_{o_{2}}=\left(0.7,0.5\right)$ , are the center of two bubbles,
$r_{1}=r_{2}=0.2$ are their radius.
\end{example}

With the prescribed data, we conduct the numerical experiments with the mesh size $h=1/64$ and time step $\tau=0.01$, and calculate the total energy ${\rm E}$ and
the mass $\int_{\Omega}\phi_{h}^{n}{\rm d}x$ at each time step. 

For a square shaped bubble, we choose the terminal time $T=1$.
Fig. \ref{fig:CubeEn} shows the evolution curve of the energy ${\rm E}$
and the mass $\int_{\Omega}\phi_{h}^{n}{\rm d}x$. We notice that
energy decay monotonically and the mass $\int_{\Omega}\phi_{h}^{n}{\rm d}x$
remains a constant. This confirms that our decoupled scheme
is energy stable and mass conservative. Snapshots of the phase evolution
at different time are presented in Fig. \ref{fig:CubeSnap}. We observe
that the isolated square shape relaxes to a circular shape, due to
the effect of surface tension. 

\noindent 
\begin{figure}
\begin{centering}
\subfloat[The energy]{\begin{centering}
\includegraphics[scale=0.5]{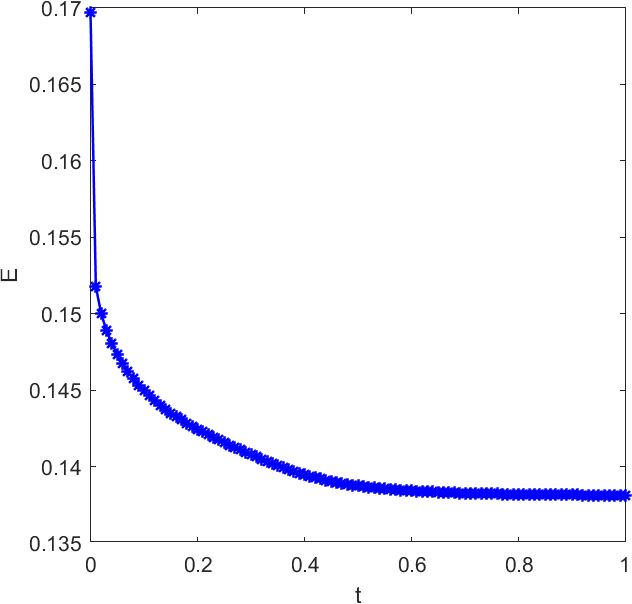}
\par\end{centering}
}\subfloat[The mass]{\begin{centering}
\includegraphics[scale=0.5]{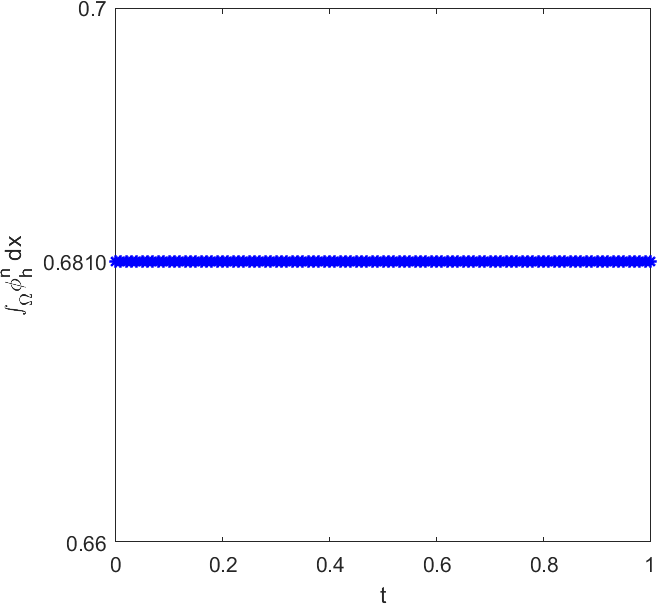}
\par\end{centering}
}
\par\end{centering}
\caption{Time evolution of the energy and mass for a square shaped bubble.\label{fig:CubeEn}}
\end{figure}

\noindent 
\begin{figure}
\begin{centering}
\subfloat[$t=0$]{\includegraphics[scale=0.22]{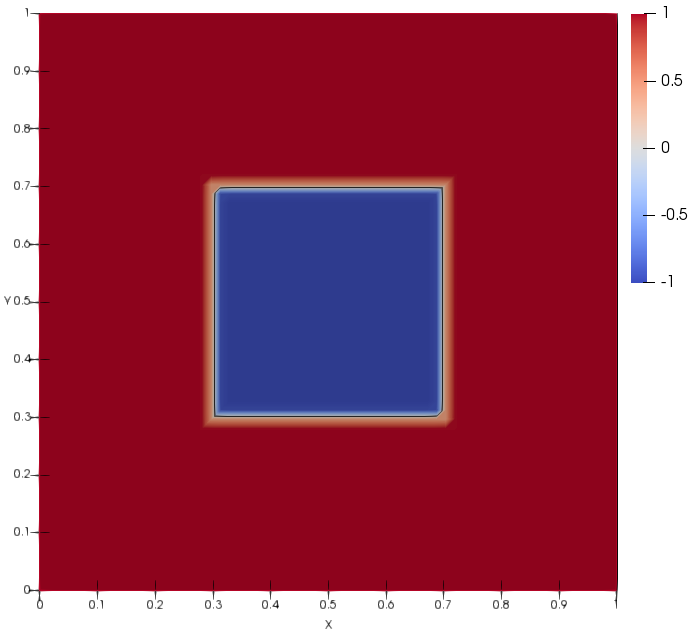}

}\subfloat[$t=0.01$]{\includegraphics[scale=0.22]{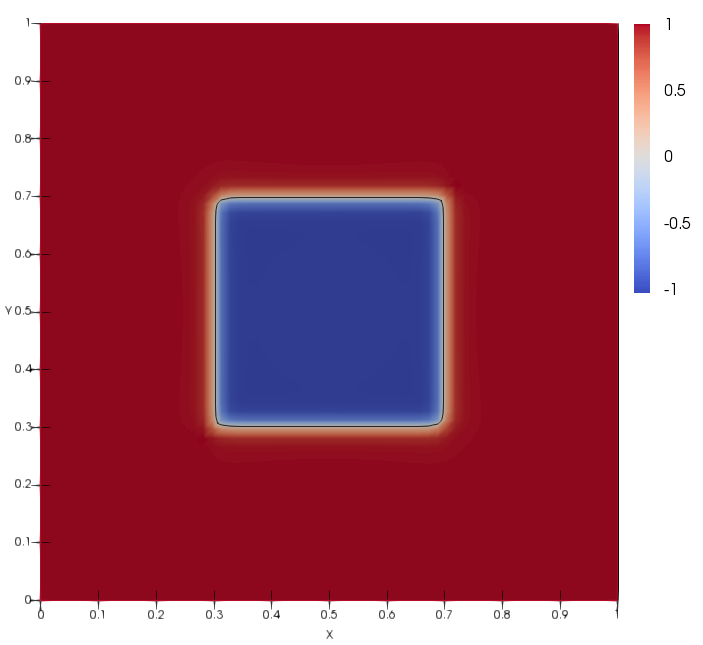}

}\subfloat[$t=0.21$]{\includegraphics[scale=0.22]{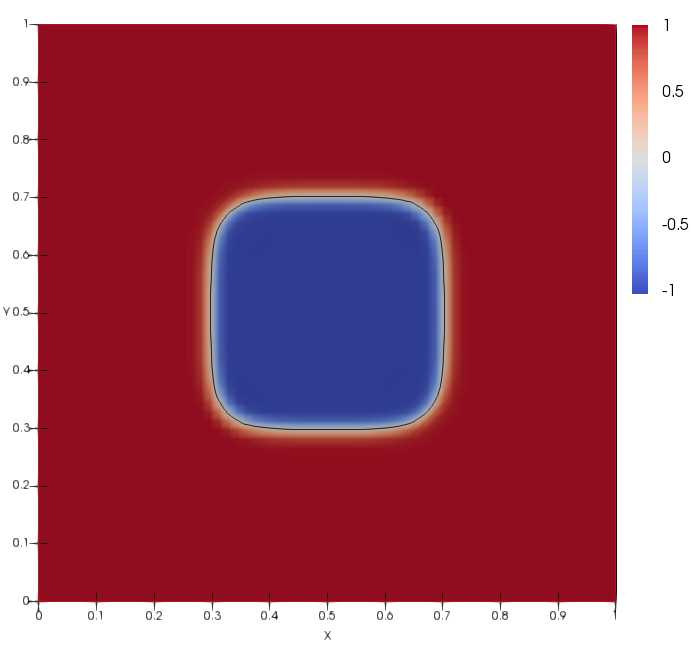}

}
\par\end{centering}
\begin{centering}
\subfloat[$t=0$.36]{\includegraphics[scale=0.22]{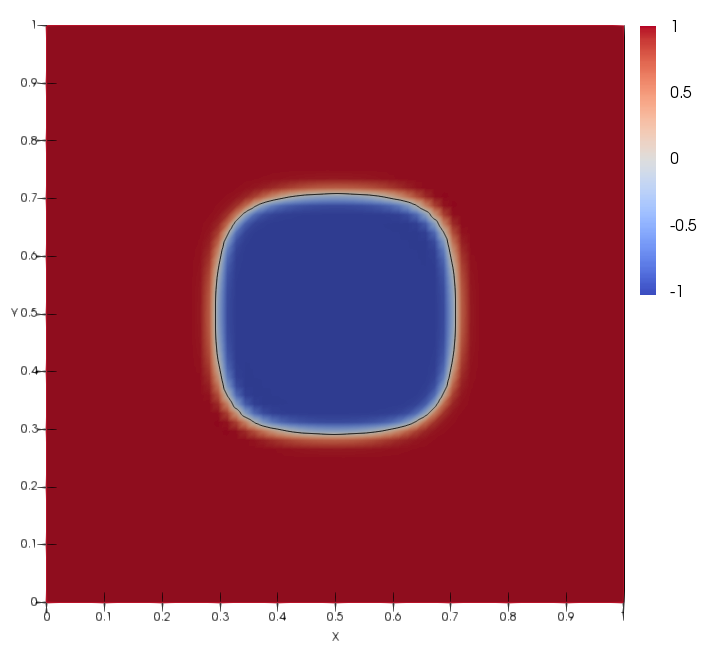}

}\subfloat[$t=0.51$]{\includegraphics[scale=0.22]{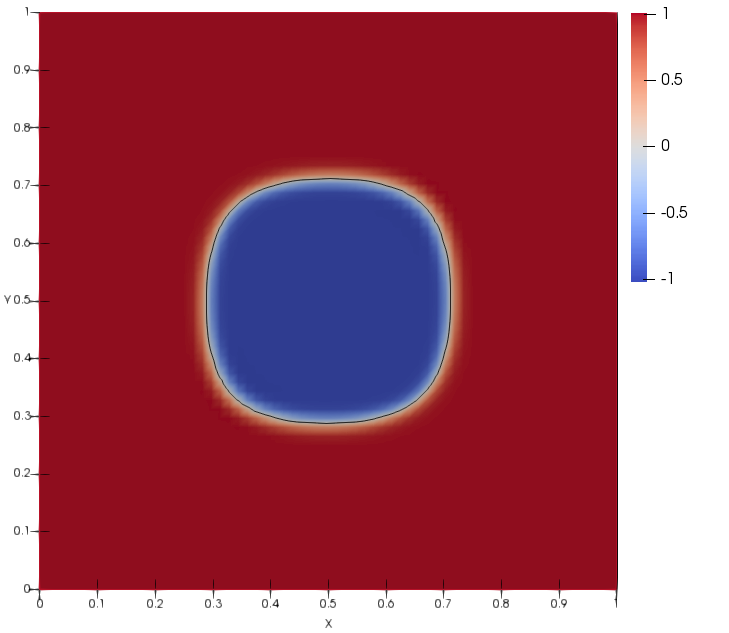}

}\subfloat[$t=1.00$]{\includegraphics[scale=0.22]{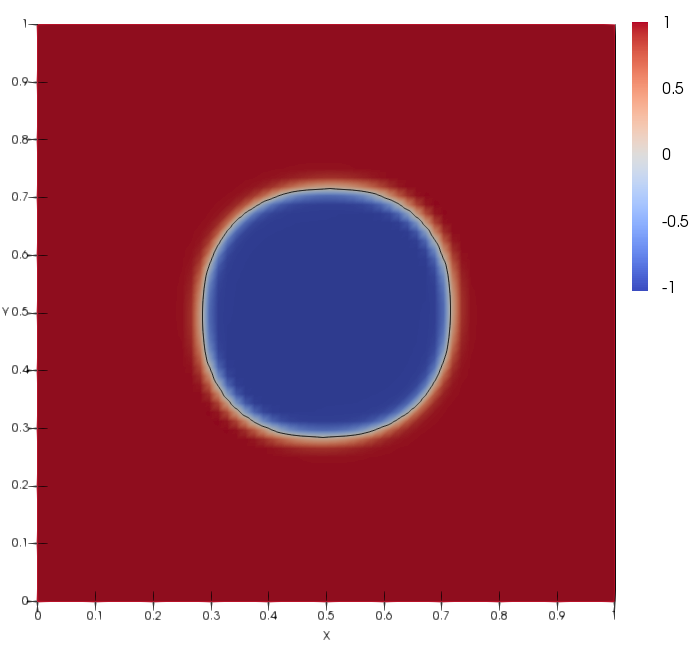}

}
\par\end{centering}
\caption{Snapshots of the relaxation of a square shape. The black circle indicates
the zero-level set of $\varphi_{h}^{n}$.\label{fig:CubeSnap}}
\end{figure}

For two kissing circular bubbles, we choose the terminal time $T=15$.
The time evolution of the energy ${\rm E}$ and the mass $\int_{\Omega}\phi_{h}^{n}{\rm d}x$
are displayed in Fig. \ref{fig:BallEn}. We still observe that energy
curves decay monotonically and the mass $\int_{\Omega}\phi_{h}^{n}{\rm d}x$
remains constant in time. Again, this confirms that the decoupled scheme
is energy stable and mass-conservative. Fig. \ref{fig:BallSnap}
displays some snapshots of the phase evolution. From this figure,
one can see that as time evolves, the two bubbles quickly connect together 
and eventually coalesces into one big bubble under the influence of
surface tension.

\noindent 
\begin{figure}
\begin{centering}
\subfloat[The energy]{\begin{centering}
\includegraphics[scale=0.5]{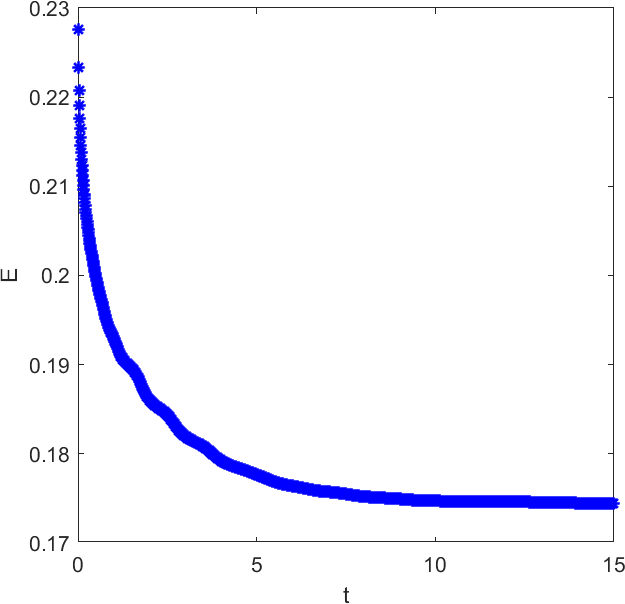}
\par\end{centering}
}\subfloat[The mass]{\begin{centering}
\includegraphics[scale=0.5]{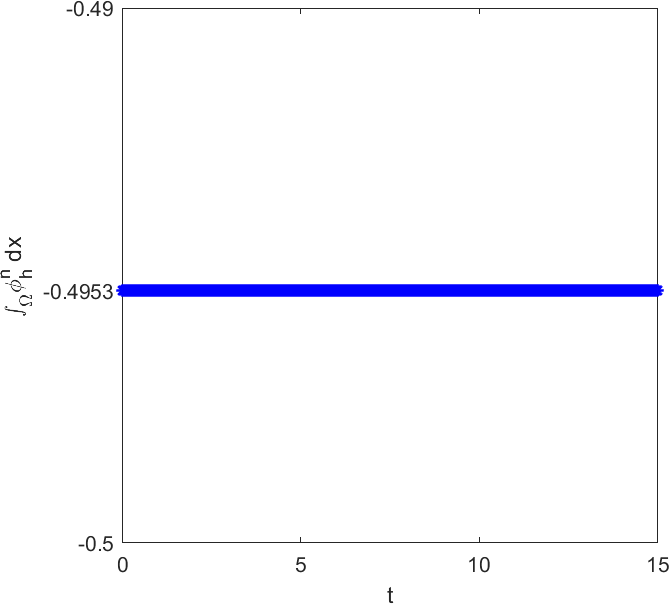}
\par\end{centering}
}
\par\end{centering}
\caption{Time evolution of the energy and mass for two kissing circular bubbles.\label{fig:BallEn}}
\end{figure}

\noindent 
\begin{figure}
\begin{centering}
\subfloat[$t=0$]{\includegraphics[scale=0.22]{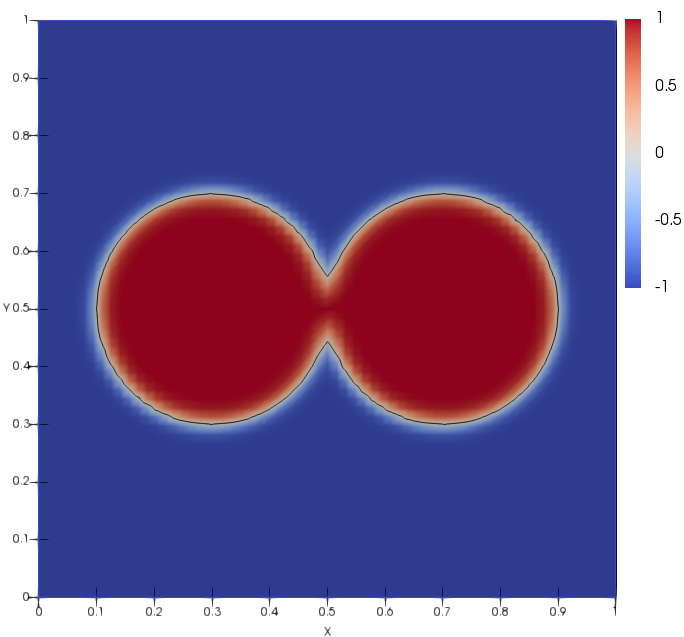}

}\subfloat[$t=0.01$]{\includegraphics[scale=0.22]{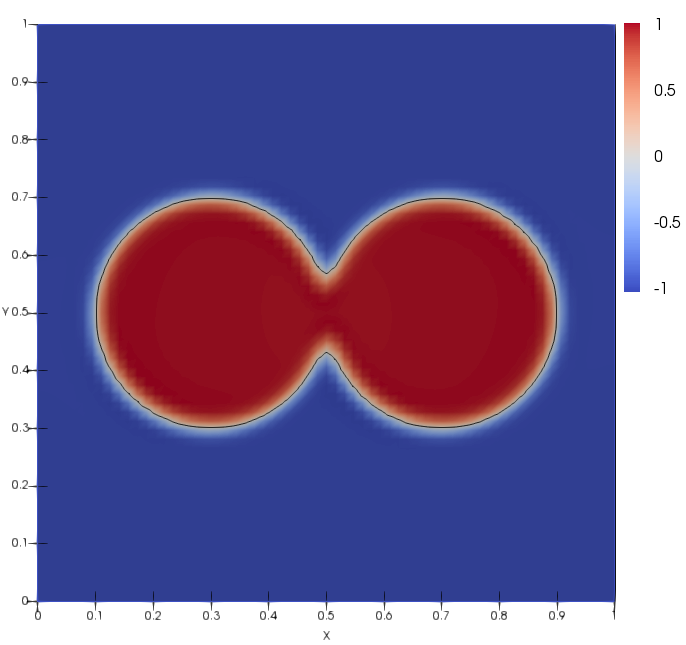}

}\subfloat[$t=0.51$]{\includegraphics[scale=0.22]{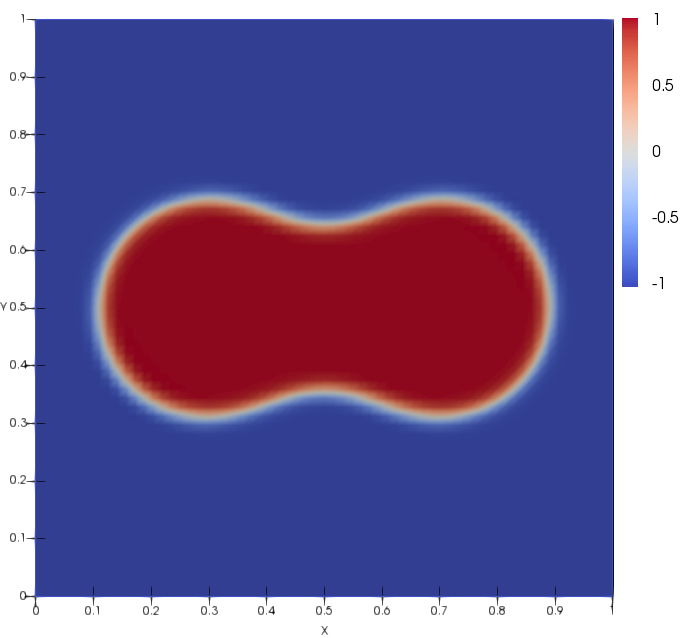}

}
\par\end{centering}
\begin{centering}
\subfloat[$t=2.51$]{\includegraphics[scale=0.22]{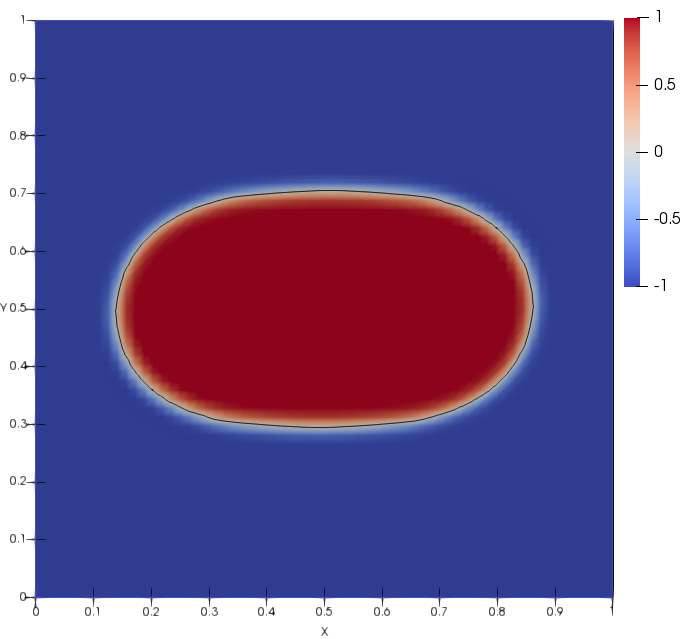}

}\subfloat[$t=5.01$]{\includegraphics[scale=0.22]{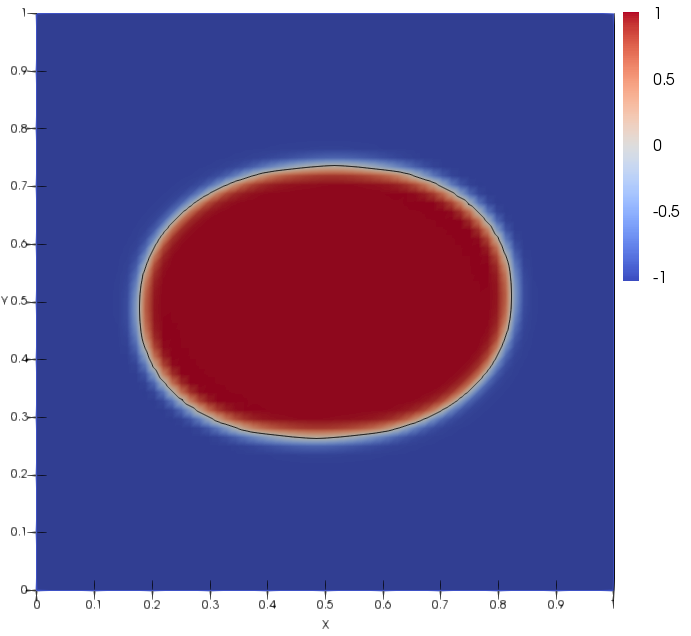}

}\subfloat[$t=7.51$]{\includegraphics[scale=0.22]{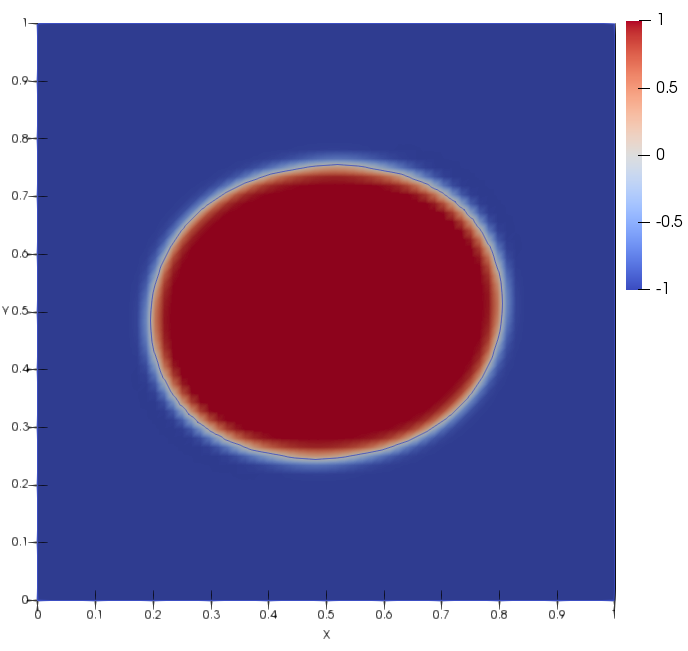}

}
\par\end{centering}
\begin{centering}
\subfloat[$t=10.01$]{\includegraphics[scale=0.22]{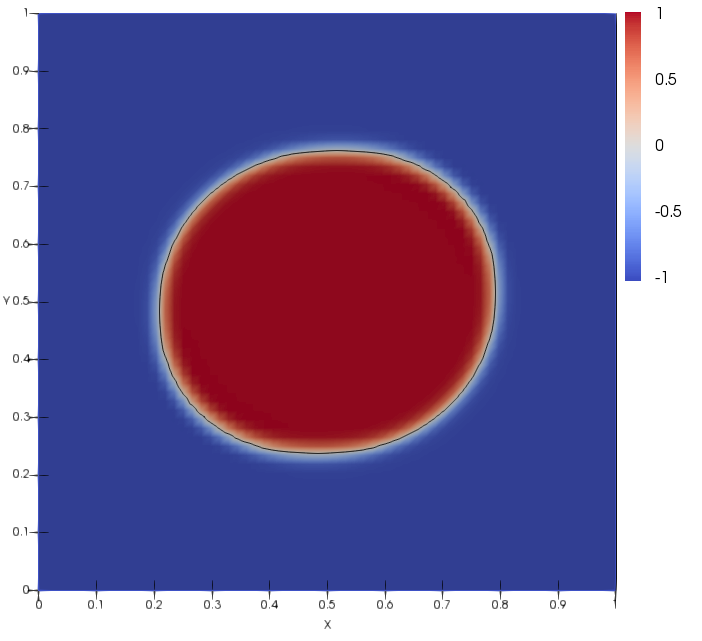}

}\subfloat[12.51]{\includegraphics[scale=0.22]{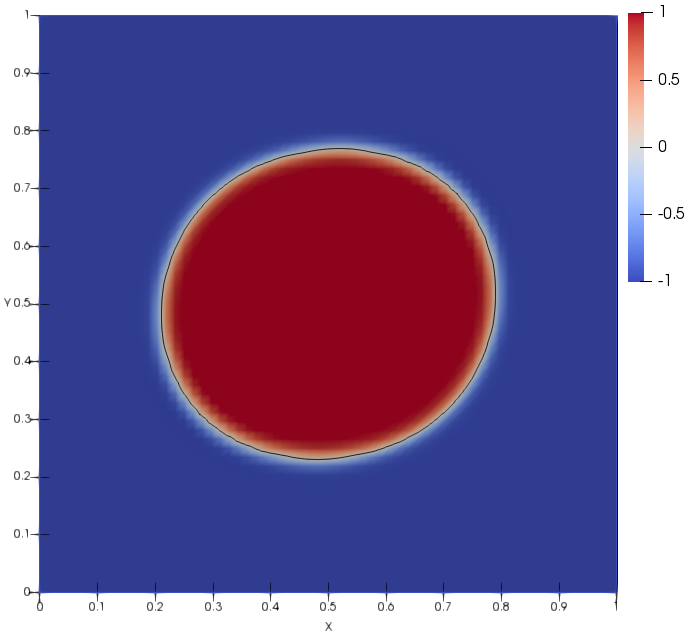}

}\subfloat[$t=15.00$]{\includegraphics[scale=0.22]{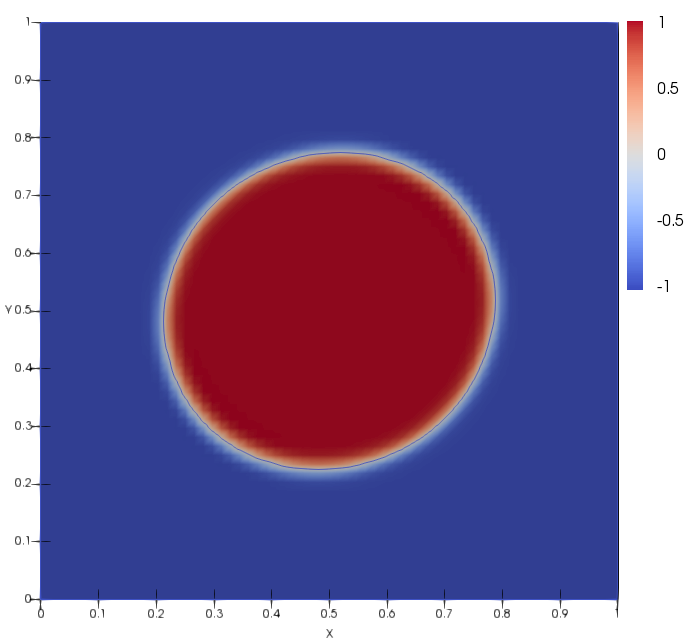}

}
\par\end{centering}
\caption{Snapshots of the relaxation of two kissing circular bubbles. The black circle indicates
the zero-level set of $\varphi_{h}^{n}$.\label{fig:BallSnap}}
\end{figure}

\begin{example}[Kelvin\textendash Helmholtz instability]
	In this example, we simulate the Kelvin\textendash Helmholtz
	instability \cite{drazin_hydrodynamic_2004}, which is one of the most fundamental instabilities in
	incompressible fluids. Considering the model problem on a rectangular
	domain $\Omega=(0,0.5)\times(0,1)$, we choose $\boldsymbol{B}=\left(0,0,1\right)^{{\rm T}}$, $\eta=0.0002,\,\sigma=1,\,\epsilon=M=0.01,$
	and $\gamma=0.001.$ The initial conditions for the phase function
	$\varphi$ and velocity $\boldsymbol{u}$ are given by 
	\[
	\ensuremath{\phi_{0}=\tanh\left(\frac{6\left(y-y_{c}\right)}{\sqrt{2}\epsilon}\right)},\quad\boldsymbol{u}_{0}=\ensuremath{\left(\tanh\left(50\left(y-y_{c}\right)\right),0\right)}^{^{{\rm T}}},
	\]
	
	\noindent where $y_{c}=0.5+0.005\sin\left(4\pi x\right)$. The boundary
	condition for velocity is 
	\[
	\boldsymbol{u}=\left(\pm1,0\right)^{{\rm T}}\quad\text{ at }y=\pm1,\quad\boldsymbol{u}=\left(0,0\right)^{{\rm T}}\quad\text{ at }x=\pm1.
	\]
\end{example}	

A uniform mesh with the step size $h=1/128$ and a uniform
time partition with the time step size $\tau=0.01$ are used in this
simulation. Fig. \ref{fig:KHSnap} shows several snapshots of the
phase field to illustrate the evolution of the shearing interface.
We observe that the flow sweeps the initial interfacial vorticity
into the center in early stage. As vorticity accumulates at the center,
the interface begins to steepen and the height of the instability
gets larger. At late stage, roll-up follows and the interface evolves
into a spiral. The dynamics of the interface is similar to those obtained
in \cite{ceniceros_threedimensional_2010,lee_khi_2015}.

\begin{figure}
	\begin{centering}
		\subfloat[$t=0$]{\includegraphics[scale=0.25]{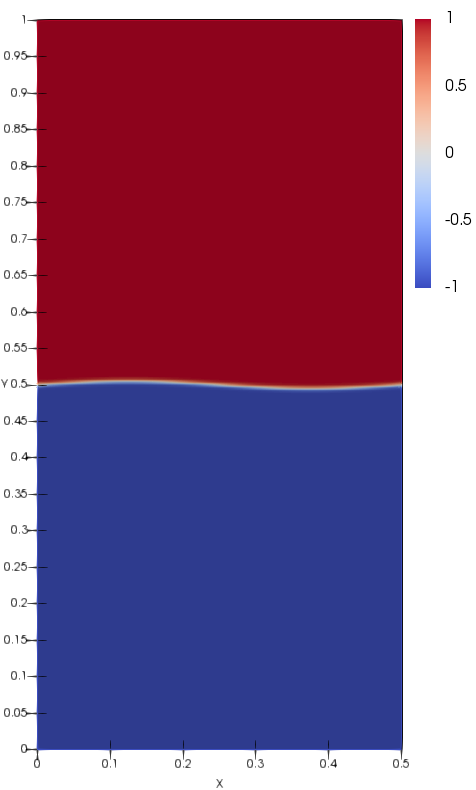}
			
		}\subfloat[$t=0.21$]{\includegraphics[scale=0.25]{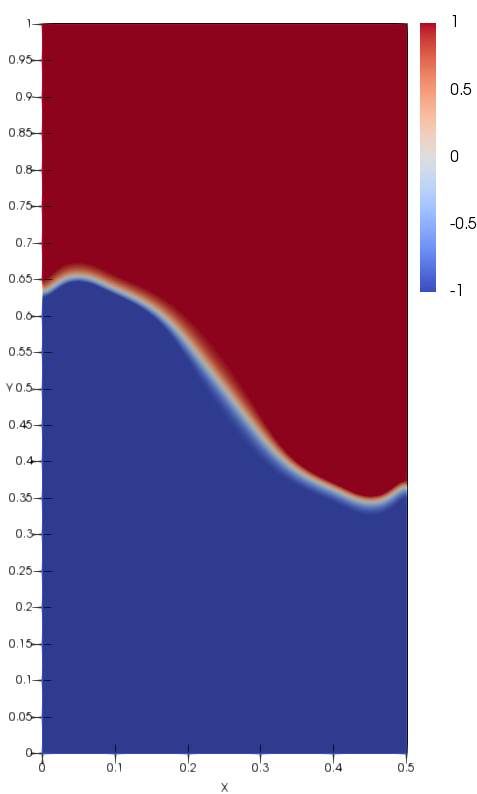}
			
		}\subfloat[$t=0.41$]{\includegraphics[scale=0.25]{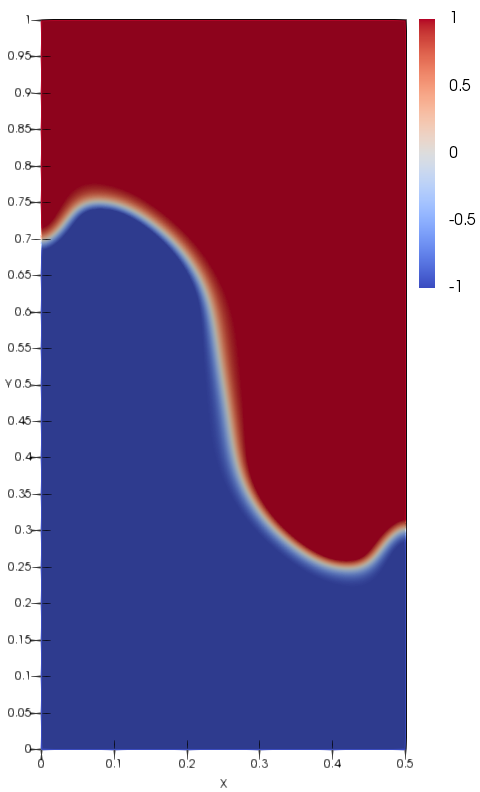}
			
		}
		\par\end{centering}
	\begin{centering}
		\subfloat[$t=0.61$]{\includegraphics[scale=0.25]{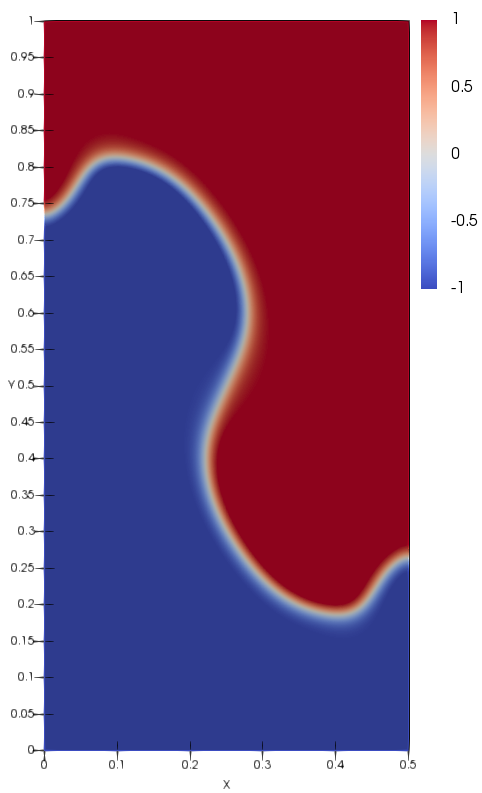}
			
		}\subfloat[$t=0.81$]{\includegraphics[scale=0.25]{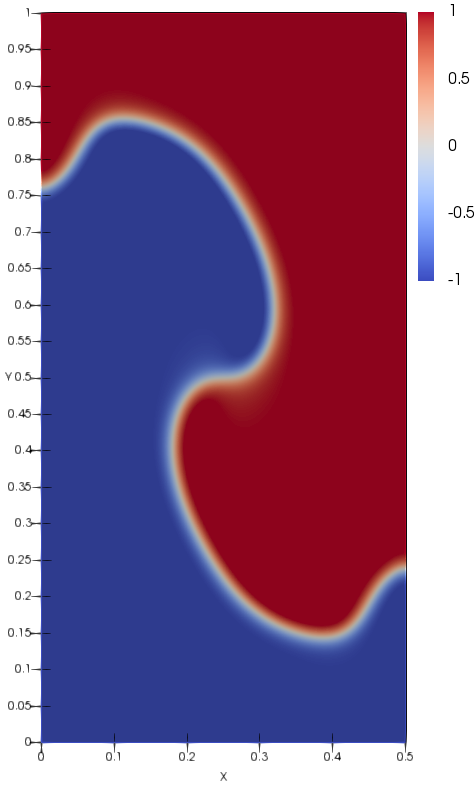}
			
		}\subfloat[$t=1.01$]{\includegraphics[scale=0.25]{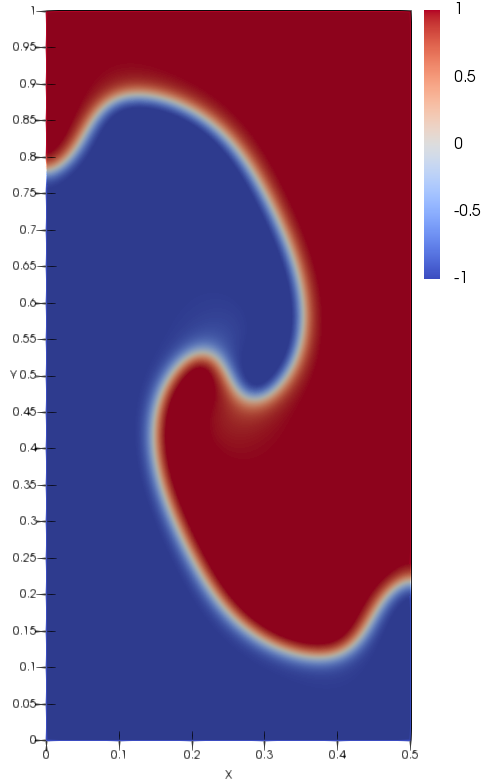}
			
		}
		\par\end{centering}
	\begin{centering}
		\subfloat[$t=1.21$]{\includegraphics[scale=0.25]{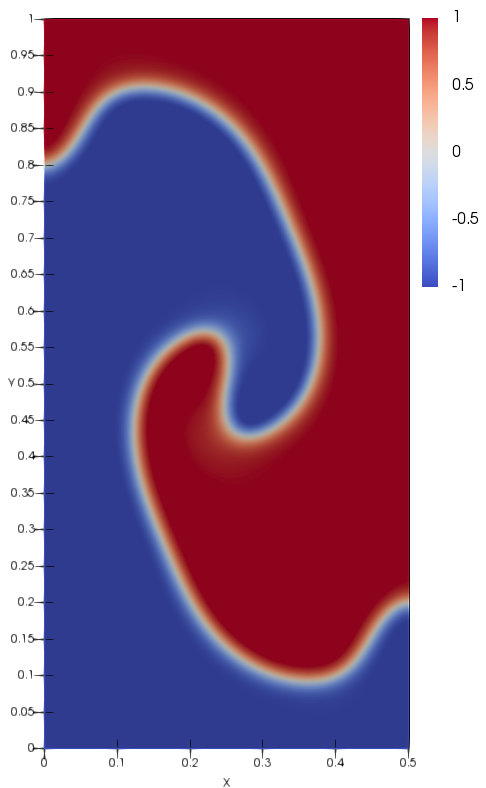}
			
		}\subfloat[$t=1.61$]{\includegraphics[scale=0.25]{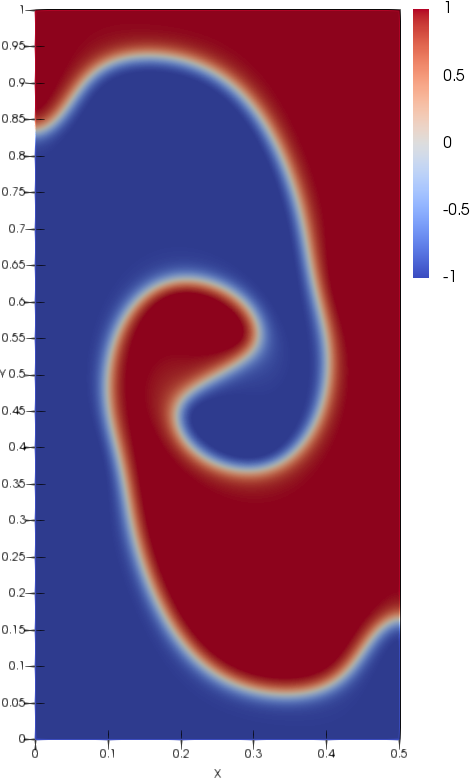}
			
		}\subfloat[$t=2.00$]{\includegraphics[scale=0.25]{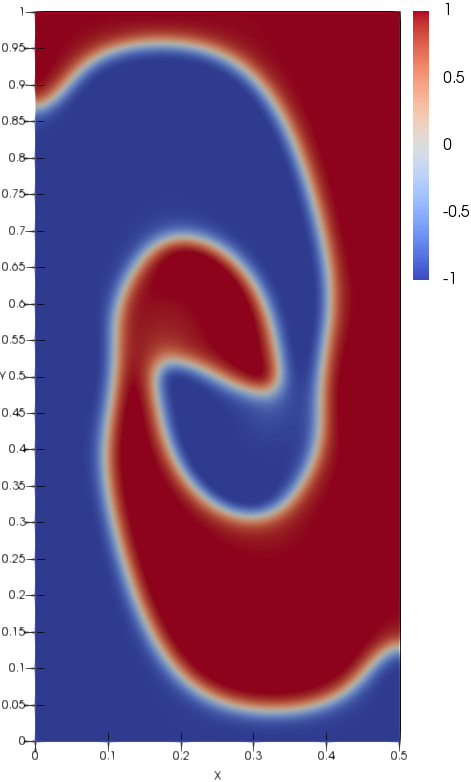}
			
		}
		\par\end{centering}
	\caption{Snapshots of the relaxation of the shearing interface. The black circle
		indicates the zero-level set of $\varphi_{h}^{n}$.\label{fig:KHSnap}}
\end{figure}

\begin{example}[Gravity-driven flow]
	This example is to study the effect of the gravity on the two-phase
	fluid. Similar to \cite{Chen2020}, we supplement the gravity
	force as a force term $\boldsymbol{f}_{g}$ on the right-hand side
	of the momentum equation (\ref{model:fluid}),
	\[
	\boldsymbol{f}_{g}=\boldsymbol{g}\frac{H\left(\varphi\right)+1}{2},
	\]
	where $H\left(\varphi\right)=\frac{1}{1+\exp\left(-\frac{\varphi}{\epsilon}\right)}$
	is a regularized approximation of the Heaviside step function, and
	$\boldsymbol{g}=(0,10)$ is the gravity. The physical parameters are
	given as $\eta=\epsilon=0.01,\sigma = 100, \,M=0.001,\ T=2.5.$ The initial data
	of $\boldsymbol{u}$ is $\boldsymbol{0}$ and of $\varphi$ is set
	as 
	\[
	\varphi_{0}=-\tanh\left(\frac{\left\Vert x-x_{o}\right\Vert -r}{\sqrt{2}\varepsilon}\right),
	\]
	where $x_{o}=\left(0.5,0.8\right)$, and $r=0.1$.
\end{example}

We simulate the CHIMHD system using the proposed scheme with $h=0.01$
and $\tau=0.005$. In Fig. \ref{fig:SnapsG}, we plot some snapshots
of the phase-field profile $\varphi$ with $\gamma=0.01,0.005,0.001$.
We observe that the effect of the gravity the circle bubble changes
shape as it falls, and it transforms to be flat when it approaches to 
the boundary. The smaller $\gamma$ is, the greater deformation.
Since the Ginzburg\textendash Landau energy models adhesion forces,
it can be expected that a reduction of the parameter $\gamma$ reduces
adhesion forces and leads to instabilities. Thus, the results verify
this point. 
\begin{figure}
		\begin{centering}
			\begin{tabular}{cccc}
				\includegraphics[scale=0.15]{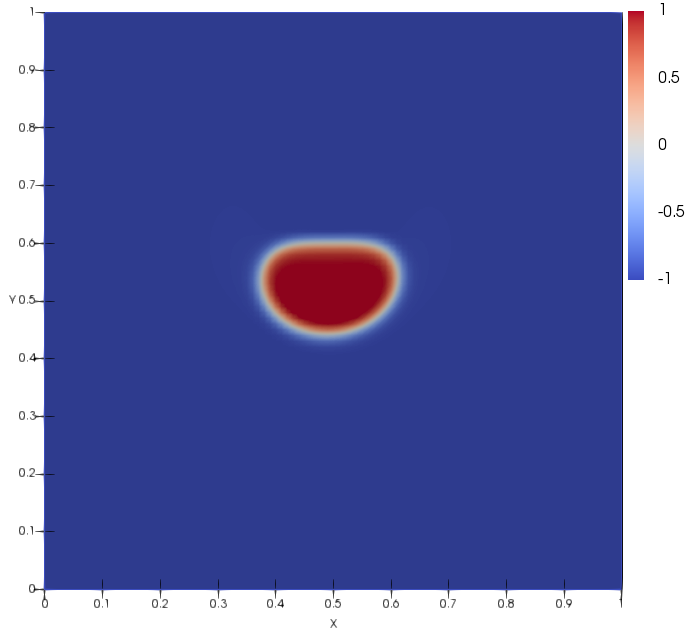} & \includegraphics[scale=0.15]{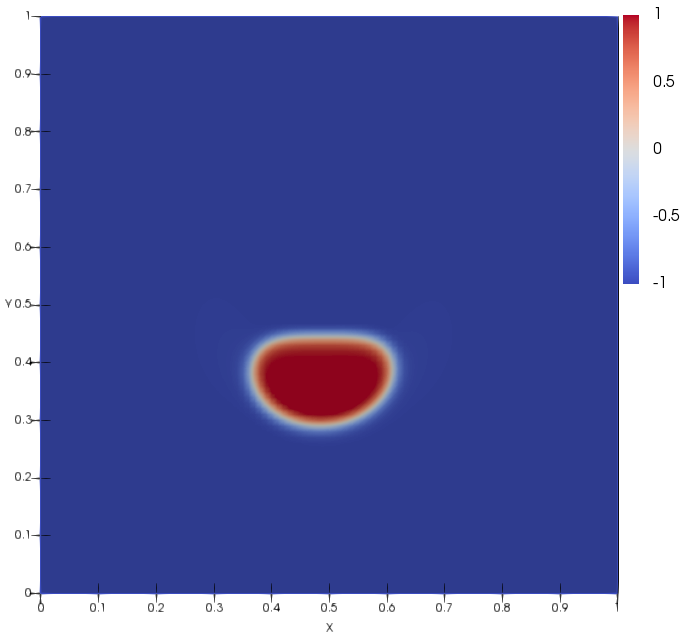} & \includegraphics[scale=0.15]{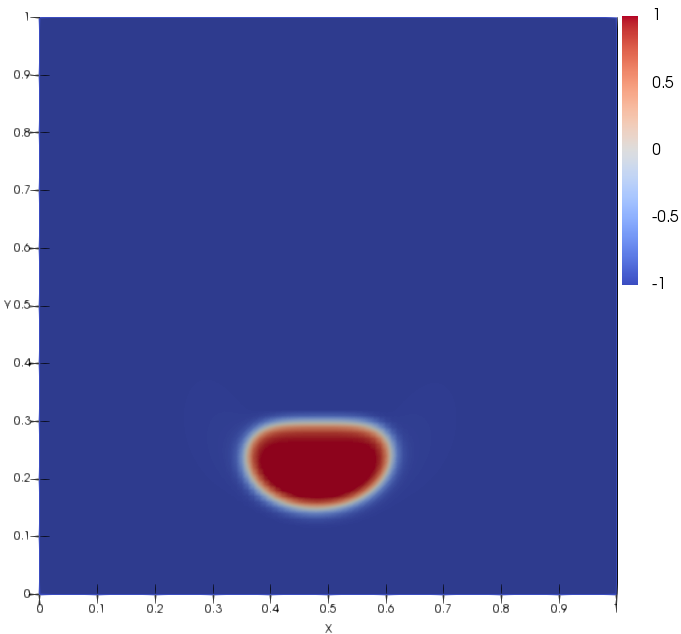} & \includegraphics[scale=0.15]{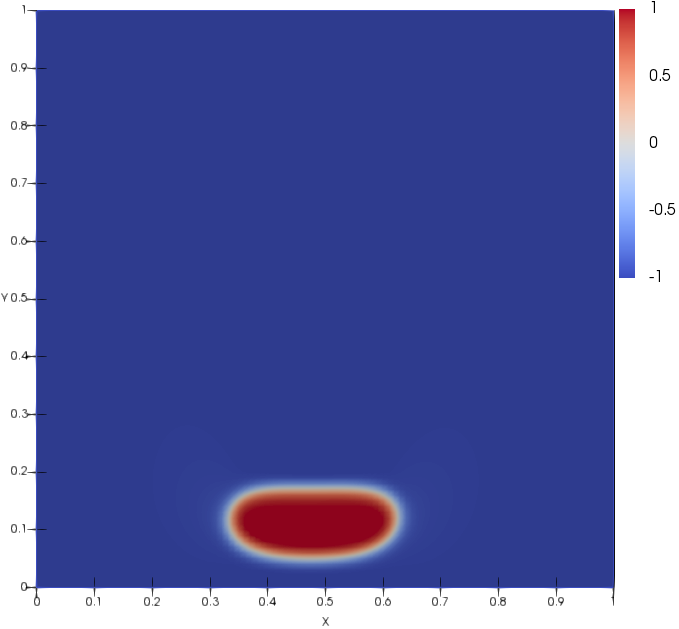}\tabularnewline
				\includegraphics[scale=0.15]{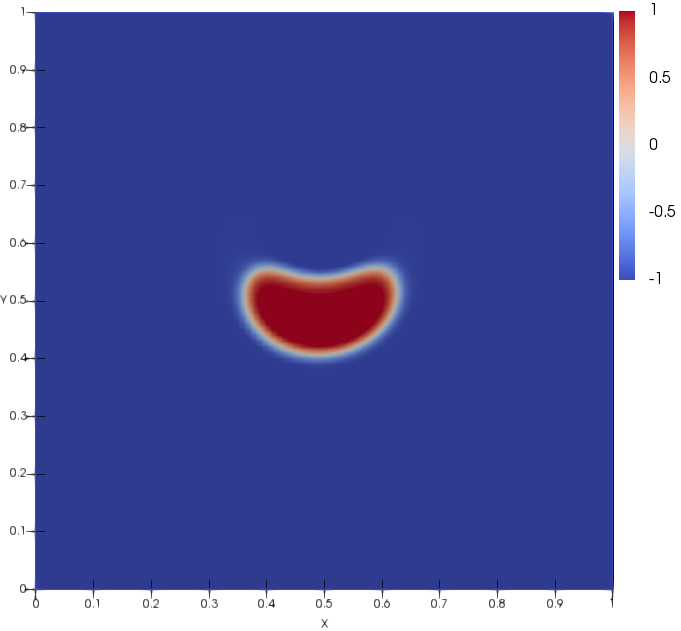} & \includegraphics[scale=0.15]{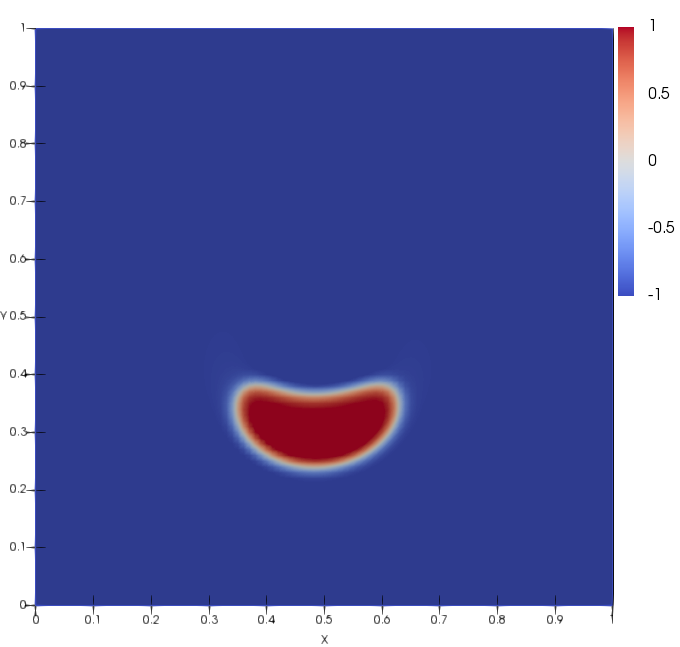} & \includegraphics[scale=0.15]{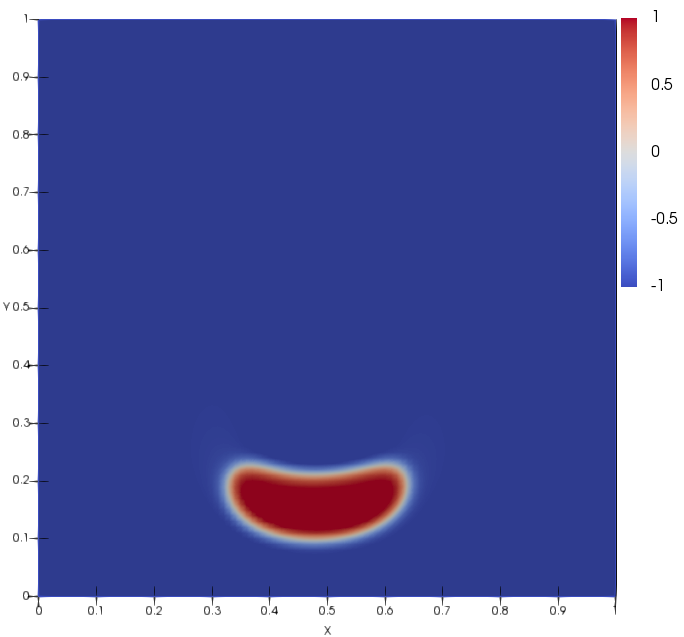} & \includegraphics[scale=0.15]{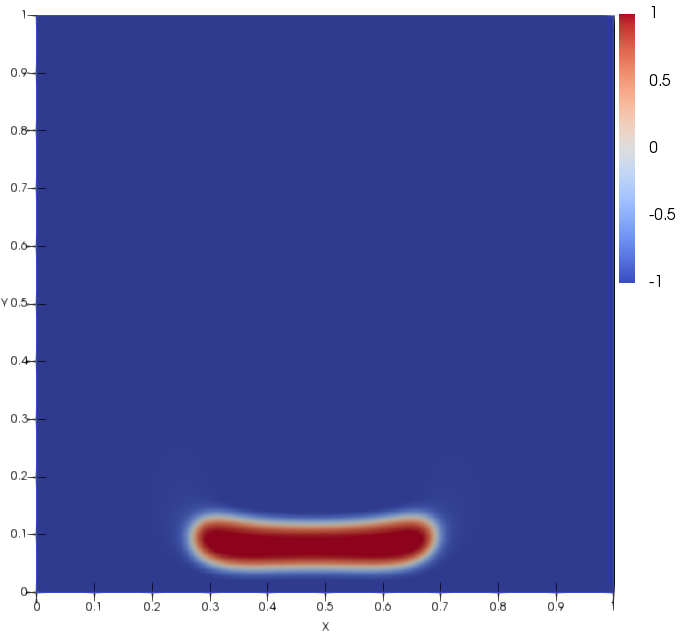}\tabularnewline
				\includegraphics[scale=0.15]{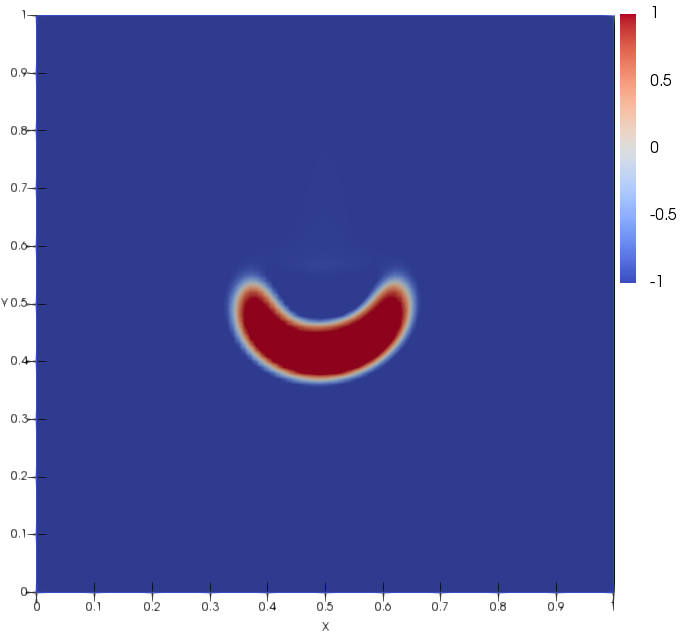} & \includegraphics[scale=0.15]{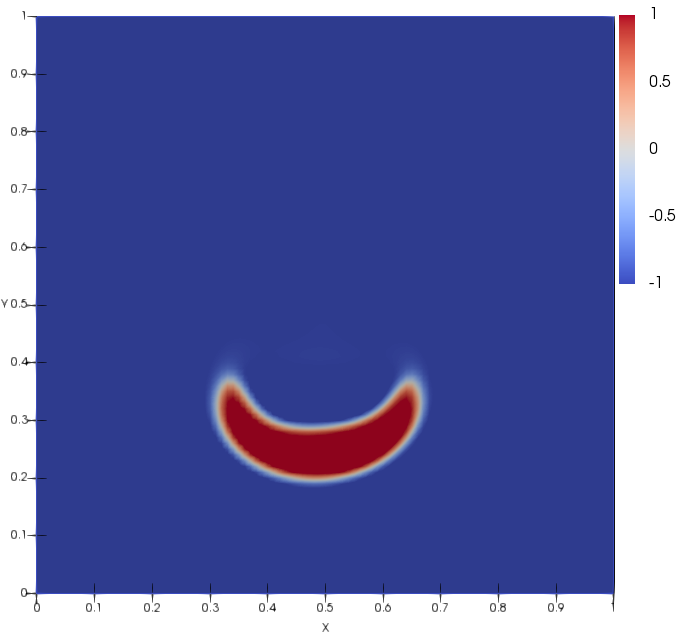} & \includegraphics[scale=0.15]{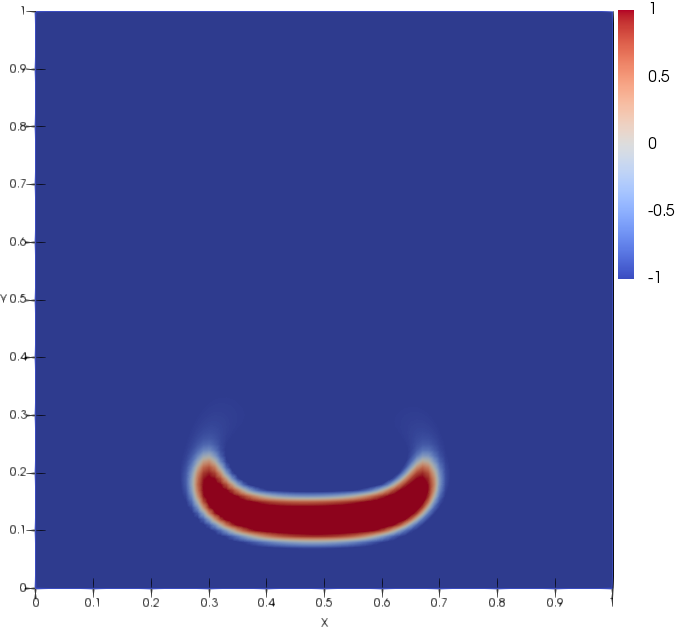} & \includegraphics[scale=0.15]{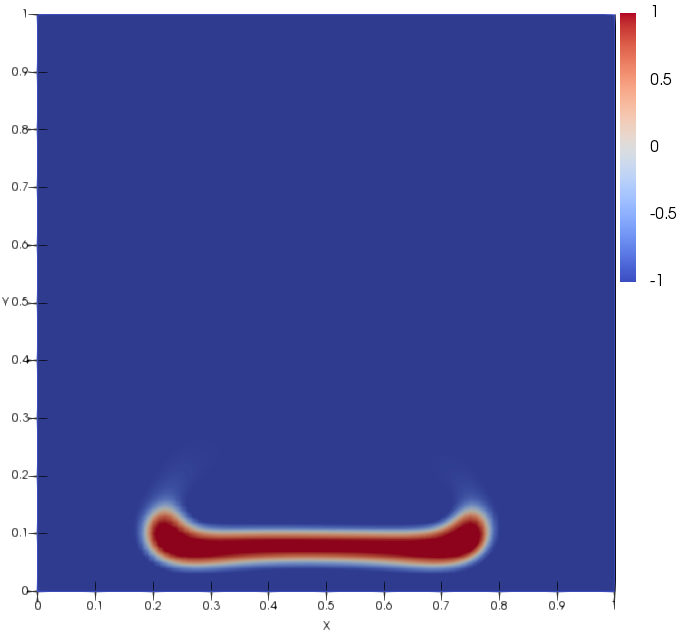}\tabularnewline
				$t=1.01$ & $t=1.51$ & $t=2.01$ & $t=2.5$\tabularnewline
			\end{tabular}
			\par\end{centering}
		\caption{\label{fig:SnapsG}Profile of the phase fields for $\gamma=0.01,0.005,0.001$.(From top to bottom.)}
\end{figure}

\section{Summary}
In this paper, we propose a linear and decoupled finite element method for the CHIMHD model. This full discrete scheme is based on first order Euler semi-implicit scheme with some first order stabilization terms and implicit-explicit treatments for time discretization, and stable mixed finite element approximation for space discretization. In particular, we solve the current density and electric potential simultaneously by using stable face-volume mixed finite element pairs to ensure the discrete current density are divergence-free exactly. The scheme is proved to be mass-conservative, charge-conservative and unconditionally energy stable. We performed some numerical tests to verify the features, accuracy and efficiency of the proposed scheme. In the further work, we will study the highly efficient scheme for the two-phase IMHD flows with large density ratio.

\section*{References}
\bibliography{1203}

\end{document}